\documentclass[11pt]{amsart}
\usepackage{amsmath,amscd,amssymb}
\usepackage[mathscr]{eucal}
\usepackage[all]{xy}

\newtheorem{theorem}{Theorem}[section]
\newtheorem{lemma}[theorem]{Lemma}

\newtheorem{corollary}[theorem]{Corollary}

\theoremstyle{definition}

\theoremstyle{remark}
\newtheorem{remark}[theorem]{Remark}

\numberwithin{equation}{section}

\def\la{\lambda}
\def\al{\alpha}

\def\e{\varepsilon}
\def\de{\delta}
\def\NN{{\mathbb N}}

\def\CC{{\mathbb C}}

\def\cT{{\mathcal T}}

\def\Re{{\rm Re}\,}
\def\Im{{\rm Im}\,}
\def\Int{{\rm Int}\,}
\def\codim{{\rm codim}\,}
\def\conv{{\rm conv}\,}
\def\diam{{\rm diam}\,}

\def\dist{{\rm dist}\,}
\def\card{{\rm card}\,}

\begin{document}
\title[Operators, bases, and matrices]{On interplay between operators, bases,\\ and matrices}

\author{V. M\"uller and Yu. Tomilov}

\address{Institute of Mathematics,
Czech Academy of Sciences,
Zitna 25, Prague,
 Czech Republic}
\email{muller@math.cas.cz}

\address{Institute of Mathematics, Polish Academy of Sciences,
\' Sniadeckich str.8, 00-656 Warsaw, Poland}
\email{ytomilov@impan.pl}

%\date{\today}

\thanks{The first author has been supported by grant No. 20-31529X of GA CR and RVO:67985840. The second author was partially supported
by NCN grant UMO-2017/27/B/ST1/00078.}

\begin{abstract}
Given a bounded linear operator $T$ on separable Hilbert space, we develop an approach allowing one to 
construct a matrix representation for $T$ having certain specified algebraic or asymptotic structure.
We obtain matrix representations for $T$ with preassigned
bands of the main diagonals, with an upper bound for all of the matrix elements, and
with entrywise polynomial lower and upper bounds for these elements. 
In particular, we substantially generalize and complement our results on diagonals of operators from \cite{MT}
and other related results.
Moreover, we obtain a vast generalization of a theorem by Stout (1981), and (partially) answer his open question.
Several of our results have no analogues in the literature.
\end{abstract}

\dedicatory{In honour of N. K. Nikolski on the occasion of his eightieth anniversary}

\maketitle

\section{Introduction: a glimpse at matrix representations}
Following conventional approach to describing operators on finite-dimen\-sional spaces as matrices,
one may represent a bounded linear operator $T$ on infinite-dimensional separable Hilbert space $H$ as the matrix 
$$A_T:=(\langle T u_n, u_j \rangle)_{j,n=1}^\infty$$
 with respect
to an orthonormal basis $(u_n)_{n=1}^{\infty} \subset H$ and try to relate the properties of 
$T$ to the properties of $A_T.$
This very natural idea looks naive to some extent,
and the study of operators on infinite-dimensional spaces through their matrix representations goes back to the birth of operator theory
in the beginning of $20$-th century,
and most notably, to Schur's multiplication and Weyl-von Neumann's perturbation theorem for self\-adjoint operators. 

While such a coordinatization  approach was neglected in favor of more revealing 
and standard by now textbook techniques,
there was still a number of interesting applications of matrix representations scattered around the literature.
Most of them are related  to the studies in this paper and serve as motivations to what follows.
So to put the paper into a proper context, we review several significant directions 
related to matrix representations of bounded operators.

\subsection{Diagonals}
For $T \in B(H)$ and an orthonormal basis $(u_n)_{n=1}^{\infty}\subset H$ the sequence  $(\langle Tu_n, u_n \rangle)_{n=1}^{\infty}$ is called
the main diagonal of $T.$
Sometimes, the word ``main'' is omitted but it will be convenient for the sequel to use this more precise terminology.
The study of diagonals of operators on infinite-dimensional Hilbert spaces and also of the related issues 
goes back
to $70-80$'s with most essential results due to Fan, Fong, and Herrero.
See e.g. \cite{Fan84}, \cite{H} and \cite{Fan87} as samples of the research from that period.
The studies got a new impetus with foundational works of Kadison and Arveson \cite{Kadison02a}, \cite{Kadison02b} and \cite{Arveson_USA}, who discovered a subtle structure 
in the set of possible main diagonals of selfadjoint projections and, more generally, normal operators with finite spectrum,
thus providing infinite-dimensional counterparts of the famous Schur-Horn theorem. 
The papers by Arveson and Kadison gave rise to a number of generalisations in various directions, including similar results for elements of von Neumann algebras, diagonals of operator tuples, applications to frame theory,
etc. See, in particular, \cite{Bownik}, \cite{Kaftal}, \cite{Kennedy}, \cite{Massey} and \cite{MT}.
As an illustration we mention the next theorem proved recently in \cite{Jasper}.

\begin{theorem}\label{jasper}
A complex-valued sequence $(d_n)_{n=1}^\infty$ is the main diagonal of a unitary operator on $H$ if and only if $\sup_{n \ge 1}|d_n|\le 1$ and
\[
2(1-\inf_{n \ge 1}|d_n|)\le\displaystyle \sum_{n=1}^\infty (1-|d_n|).
\]
\end{theorem}

A nice survey of recent developments in the theory of operator diagonals can be found in \cite{LW20}.
See also  \cite{MT} and a discussion of some recent applications of numerical ranges in \cite{MT_surv}, including the references therein.

In \cite{MT}, we have changed a perspective by describing the main diagonals for a given operator $T$ rather
than the set of possible main diagonals for operator classes.
Among other things, it was proved in \cite{MT} that if the essential numerical range of a bounded operator $T$ on $H$ has a non-empty interior,
then  a sequence from the interior 
is the main diagonal of  $T$ if it approaches the boundary not too fast, satisfying so-called Blaschke-type condition.
Such a condition is often optimal. For a more detailed discussion of some of the results from \cite{MT}, see Section \ref{results} below.

It seems that the methods of \cite{MT} opens a much wider venue than the one sketched in \cite{MT},
and we hope the results of this paper justify this claim.

\subsection{Banded matrices and matrices with special structure}
One of the basic advantages in dealing with operator matrices is that for several important classes of operators 
their matrices have so-called banded structure. Recall that for $n \in \mathbb N$
an operator is said to be $n$-diagonal if it is
 unitarily equivalent to a (finite or infinite) direct sum of (finite or
infinite) $n$-diagonal matrices. The $n$-diagonal operators are often called band-diagonal
when particular value of $n$ is not crucial.

It is well known (and easy to prove) that selfadjont operators are $3$-diagonal.
This fact constitutes a basis for the classical approach to the study of selfadjoint (mostly unbounded) operators
via associated $3$-diagonal Jacobi matrices. However, any $3$-diagonal unitary operator is diagonal.
In fact, any unitary operator is $5$-diagonal, and this number of diagonals is optimal.
The relevance of this fact for mathematical physics  
was recognized comparatively recently, mainly due to so-called CMV-representations developed in \cite{CMV} and \cite{CMV1},
see also \cite{Simon} for an additional insight. Of course, not only the mere fact of three or five diagonality of the matrix, but also the availability of concrete and convenient basis is important here.

However, band-diagonality is quite a rare phenomenon.
In particular, as proved in \cite{Shulman}, see also \cite{FoWu96}, every normal
operator with spectral measure not supported on a set of planar measure zero
is not band-diagonal, and so, in particular, the multiplication operator $(M f)(z)= z f(z)$ on $L^2(\mathbb D)$ on the unit disc $\mathbb D$
is not band-diagonal. Moreover, there are non-band-diagonal operators in the intersection of all Schatten $p$-classes for $p>2$,  and the set of all non-band-diagonal operators is dense in the space $B(H)$ of bounded linear operators on $H$.
At the same time the set of band-diagonal operators is not norm-dense in $B(H)$ being quite a small subset
of $B(H)$ in various senses. For a discussion of restrictions posed by band-diagonality
from the point of view of $C^*$-algebras, see e.g. \cite[Chapter 16]{BrownO}. 
An illuminating discussion of band-diagonality can be found in \cite{FoWu96}.

Another closely related topic concerns universal matrix representations with sparsified structure,
that is representations possessing many zeros. It seems such representations go back
to \cite{Weiss}, where they were called staircase representations.
One may prove that any $T \in B(H)$ admits a universal block three-diagonal
form with the exponential control on (finite-dimensional) block sizes:
\begin{equation*}\label{repres}
T = \begin{pmatrix}
D_1 &  U_1  &0 & \ldots      \\
L_1 & D_2& U_2&\ddots         \\          
 0 &L_2 &D_3 &\ddots           \\
\vdots &\ddots &\ddots &\ddots  
\end{pmatrix}.
\end{equation*}

 The representations are useful in commutator
theory, e.g. in the study of Pearcy-Topping problem on compact commutators. Their modern and pertinent discussion 
can be found in \cite{Pat}.  As examples of other applications of the three diagonal block representaions we mention
\cite{Muller_Studia}, where the Olsen lifting problem was treated, and \cite{FoWu93},
addressing representations of operators in $B(H)$ as linear combinations of operators of simple form (e.g. diagonal).  

The same issues of sparsifying and arranging certain arrays of zero (and not only) elements
in finite matrices still form a vast area of research. Being unable to give any
reasonable account we refer to \cite{HoSc07} and \cite{DoJo07} as an illustration of problems and approaches considered there.

Apart from  matrices with banded structure there was a research on finding a basis  $(u_n)_{n=0}^{\infty}\subset H$  such that the matrix $A_T:=(\langle T u_n, u_j \rangle)_{j,n=0}^\infty$
has a special canonical form. 
For instance, it would be instructive to recall that for any $T \in B(H)$ the property of cyclicity  can be recasted as the existence of $A_T$ of a special ``triangular$+1$'' form,
see \cite[p. 285-286]{Halmos-book}, and e.g. \cite[Theorem 5]{Douglas} for a related result.
Being unable to mention all of the relevant papers, we only note a fundamental work \cite{MPT} where selfadjoint $T \in B(H)$ having a Hankel matrix  in some  $(u_n)_{n=0}^{\infty}$ (i.e. $\langle T u_n, u_j \rangle=\alpha_{n+j}$ for a sequence $(\alpha_n)_{n=0}^{\infty}\subset \mathbb R_+$) were characterized in spectral terms. The criteria is rather demanding and thus very
different from our assumptions, which are minimal in a sense. 

\subsection{Big matrices}

One of the basic results in the classical harmonic analysis, due to de Leeuw, Kahane, and Katznelson,
says that for any $(a_n)_{n\in \mathbb Z} \in \ell^2(\mathbb Z)$ there exists a periodic $f \in C([0,2\pi])$ such that the Fourier coefficients
$(\hat f_n)_{n \in \mathbb Z}$ of $f$ satisfy
$|\hat f_n|\ge |a_{n}|, n \in \mathbb Z,$
so that $L^2([0,2\pi])$-functions are indistinguishable from $C([0,2\pi])$-functions by the size of their
Fourier coefficients. Later on, numerous extensions of this result were found for other classes of $f$ including also
some spaces of functions analytic in $\mathbb D,$ see e.g \cite{Ball} for the overview of statements of that flavor,
 called plank type
in view of similarity of the employed methods to the ones in Bang's theorem on covering convex body by planks.

On the way to obtaining noncommutative counterparts of the above domination result, Lust-Piquard
proved in \cite{Lust} the next elegant theorem on a possible size of matrices of bounded Hilbert space operators (which can be formulated for matrices in both $\mathbb Z$- and $\mathbb N$-settings, and we prefer the latter convention).
\begin{theorem}\label{lust}
For every matrix of complex numbers $A=(a_{jn})_{j,n=1}^\infty$  such that
\begin{equation}\label{necess}
\|A\|_{\ell_\infty(\ell_2)}:=\sup_{n \ge 1}(\sum_{j \ge 1} |a_{nj}|^2)^{1/2}<\infty \qquad \text{and}\qquad  \|A^*\|_{\ell_\infty(\ell_2)}<\infty,
\end{equation}
where $A^*$ stands for the conjugate transpose of $A,$ there exists $T\in B(H)$ and an orthonormal basis $(u_n)_{n=1}^{\infty}\subset H$ such that 
\begin{equation}\label{plank1}
\|T\| \leq K\max\{\|A\|_{\ell_\infty(\ell_2)},\|A^*\|_{\ell_\infty(\ell_2)}\} \qquad \text{and} \qquad 
|\langle Tu_n, u_j \rangle|\geq |a_{nj}|
\end{equation}
for all $n,j \ge 1$, where $K$ is an absolute constant.
\end{theorem}
Clearly, the assumption  \eqref{necess} is necessary for any estimates as \eqref{plank1} to hold,
since \eqref{necess} is satisfied by $A_T=(\langle Tu_n, u_j \rangle)_{j,n=1}^\infty$ in view of $T \in B(\ell^2(\mathbb N)).$
Apart from being instructive as such, Theorem \ref{lust} appeared to be crucial in the characterization of wide classes of 
Schur multipliers on $B(\ell^2),$ see e.g. \cite{Davidson}.

\subsection{Small matrices}

Given a basis $U=(u_n)_{n=1}^\infty \subset H$ and  $T, S \in B(H),$  the entrywise product  $A_T*A_S$ of $A_T$ defines  $R \in B(H)$ such that $A_R = A_T * A_S$ and $\|R\|\le \|T\| \|S\|.$ With such a product, called Schur multiplication,
 $B(H)$ becomes a commutative Banach algebra $B_U(H),$  called respectively Schur algebra. Evidently, the properties of $B_U(H)$ are related
to the choice of $U$, and this relation is highly non-trivial. 
The Schur algebras, along with related notion of Schur multiplier, are among domains of research, 
where matrix representations appear naturally as a part of the very first definitions and surrounding basic results.
Without going into details of this rather involved area, we emphasize a result not referring to specific notions and 
important for our further considerations.

In his study of the Schur multiplication on $B(H)$,  Stout discovered in \cite{Stout} that if $\langle T e_n, e_n\rangle \in c_0(\mathbb N)$ for some orthonormal set $(e_n)_{n=1}^{\infty},$
i.e., $0$ belongs to the essential numerical range $W_e(T)$ of $T$ (in fact, Stout used a different definition of $W_e(T)$),  then the size of matrix elements $\langle Tu_n, u_j\rangle$ can also be made small
for an appropriate basis $(u_n)_{n=1}^{\infty}.$ More precisely, the next theorem holds, see \cite[Theorem 2.3]{Stout}.
\begin{theorem}\label{stout}
Let $T \in B(H).$ Then the following properties are equivalent.
\begin{itemize} 
\item [(i)] $0\in W_e(T);$
\item [(ii)] For every $\epsilon >0$ there exists an orthonormal basis $(u_n)_{n=1}^\infty$ such that $|\langle T u_n, u_j \rangle|<\epsilon$ for all $n,j \in \mathbb N;$ 
\item [(iii)] For any  $(a_n)_{n=1}^\infty\subset (0,\infty)$ satisfying  $(a_n)_{n=1}^\infty\not \in \ell^1(\mathbb N)$ there exists an orthonormal basis $(u_n)_{n=1}^\infty$ in $H$ such that
\begin{equation}\label{iii}
|\langle T u_n, u_n \rangle| \le a _n 
\end{equation}
 for all $n \in \mathbb N.$
\end{itemize}
\end{theorem}

For motivation of Theorem \ref{stout}, and its relation to the structure of Schur algebars, see \cite{Stout} and \cite{Stout1}.
The statements similar but slightly weaker than Theorem \ref{stout} appeared to be crucial for the study of various matrix norms on $B(H),$
in particular for comparing them to each other and also to the operator norm on $B(H).$ See e.g. \cite{FoRaRo87} for this direction of studies.

\subsection{Order properties}

Quite often, the order properties of matrix elements prove to be useful. In particular, Radjavi and Rosenthal showed in \cite{RaRo}  that for any $T \in B(H)$
there exists a selfadjoint $S \in B(H)$ such that $T$ and $S$ have no common invariant subspaces, and as a consequence 
the operators $T$ and $S$ generate $B(H).$ A key step in their approach is to note that
if $T \in B(H)$ is not a multiple of the identity, then there exists an orthonormal basis $\{u_n\}_{n=1}^{\infty}$ 
such that $\langle T u_n,u_j \rangle  \neq 0$ for all $n$ and $j$. In fact, the statement is true even for a sequence of bounded operators $(T_k)_{k=1}^\infty$ on $H$ rather than a single operator $T$. This matrix statement was used frequently in similar contexts. 

In the study of cyclicity and multi-cyclicity phenomena, the next observation due to Deddens played an important role:
if for $T \in B(H)$ there exists a basis $(u_n)_{n=1}^{\infty}$ in $H$ such that  $\langle T u_n, u_j \rangle$ are real for all 
$n$ and $j \in \mathbb N$, then $T \oplus T^*,$ where $T^*$ is adjoint of $T,$ is not cyclic. In particular, if $S$ is the unilateral shift on $H^2(\mathbb D),$ 
then $S\oplus S^*$ is not cyclic. For more on that see [26] (and also \cite[p. 283]{Halmos}), though rather unfortunately bases leading 
to ``real'' matrix representations of $T$  have not been studied subsequently in the literature.

\subsection{Halmos problem}

The next problem appeared while developing the theory of integral operators, but it is quite natural as such, e.g. 
in the study of Schur multiplication and associated Schur operator algebras on Hilbert spaces, see e.g. \cite{Stout}.
Let us say that $T \in B(H)$ is absolutely bounded if the matrix $|A_T|:=(|\langle Tu_n, u_j \rangle|)_{j,n=1}^\infty$
defines a bounded operator on $\ell^2(\mathbb N)$ for some orthonormal basis $(u_n)_{n=1}^\infty$ in $H,$ and totally absolutely bounded if 
$|A_T| \in B(\ell^2(\mathbb N))$ for any orthonormal basis $(u_n)_{n=1}^\infty$ in $H.$
Clearly, not every $T \in B(H)$ is absolutely bounded. The simplest example is probably
$T$ on $\ell^2(\mathbb N)$ given by the so-called Hilbert matrix ($a_{j,n})_{j,n \ge 1},$ where $a_{jn}=(j-n)^{-1}$ for $n \neq j$ and $a_{nn}=0$,
see e.g. \cite[Example 10.1]{Halmos_S}. In \cite{Halmos} Halmos asked for a characterization of absolutely bounded and totally absolutely bounded $T$  (see \cite{Halmos_S} for more on these notions and their motivations.) Independently and almost simultaneously, it was proved in \cite{So78} and \cite{Su78} that $T$ is totally absolutely bounded
if and only if $T=\lambda + S,$ where $\lambda \in \mathbb C$ and $S$ is a Hilbert-Schmidt operator. 
However, the description of absolutely bounded $T$ is still out of reach, though apart from Halmos, the problem was posed explicitly 
in  \cite{So78}, \cite{Stout}, \cite{Su78},  and \cite{Su81}. A discussion of this and related matters can be found in \cite[Chapter 10 and Theorem 16.3]{Halmos_S}. 
Obviously, the problem reflects the current lack of understanding on how the entries of $A_T$ may change when $(u_n)_{n=1}^{\infty}$ is varying. Another illustration of this problem is the study of operator diagonals discussed above, 
where the explicit description of the set of all diagonals  is known
for very particular choices of $T,$ even if $T$ is selfadjoint,
let along the description of such a set 
%of diagonals 
for fixed $T.$

\section{Tools, results and strategy}\label{results}
\subsection{Numerical ranges}
In our studies of matrix representations for a bounded operator $T$ on a separable (complex)  infinite-dimensional Hilbert space $H$
we will rely on elementary  properties of its numerical range $W(T):=\{\langle Tx, x\rangle: \|x\|=1\}$ and its essential numerical range
$W_e(T).$ The main property of  $W(T)$ is that it is always a convex subset of $\mathbb C$, and moreover the spectrum $\sigma (T)$ of $T$ is contained in the closure
$\overline{W(T)}$ of $W(T).$ However, $W(T)$ can be neither closed nor open.

Being only an approximate analogue of $W(T),$ the essential numerical range possess however much nice properties.
Recall that for $T \in B(H)$ its essential numerical range $W_e(T)$ can be defined
as
\begin{equation}\label{we}
W_e(T):=\{\lambda \in \mathbb C: \langle T e_n, e_n \rangle \to \lambda, \,\, n \to \infty\}.
\end{equation}
for some orthonormal sequence $(e_n)_{n=1}^{\infty}\subset H.$
In fact, an orthonormal sequence in \eqref{we} can be replaced by an orthonormal basis, see e.g. Theorem \ref{stout}, (ii).
Alternatively, $\lambda \in W_e(T)$ if and only if for every $\epsilon >0$ and every subspace $M$ of finite codimension there exists a unit vector $x \in M$ such that
$|\langle Tx, x \rangle-\lambda|<\epsilon.$ 
Clearly, $W_e(T) \subset \overline{W(T)}.$
For any $T \in B(H),$ the set $W_e(T)$ is non-empty, compact and convex, and moreover, $W_e(T)$ contains the essential spectrum $\sigma_e(T)$ of $T.$
Thus, in view of the convexity of $W_e(T),$ 
\begin{equation}\label{spectrum}
{\rm conv} \sigma_e(T)\subset W_e(T),
\end{equation}
where ${\rm conv} \sigma_e(T)$ denotes  
the convex hull  of $\sigma_e(T).$
Since  for any contraction $T$ with $\sigma(T)\supset \mathbb T$ (e.g. unilateral or bilateral shift) one has $\sigma_e(T)\supset \mathbb T,$ from \eqref{spectrum} it follows that
$\mathbb D$ belongs to the interior $\Int \, W_e(T)$ of $W_e(T).$ So $\Int \, W_e(T)$ is as large as possible in this case.
The convexity of $W(T)$ also implies that ${\rm Int}\, W_e(T)\subset W(T).$ Since
$W_e(T)=W_e((I-P)T(I-P))$ for every finite rank projection $P,$ one has 
\begin{equation}\label{inc_w}
{\rm Int}\, W_e(T) \subset W((I-P)T|_{(I-P)}). 
\end{equation}
Thus, to be able to find a fixed   $\lambda \in W(T)$ in the numerical range of any restriction of $T$ to a finite-codimensional subspace of $H,$ it is natural
to  assume that $\lambda \in {\rm Int}\, W_e(T).$
(Though the latter is not necessary for the former as the example of the zero operator shows.)

The  property \eqref{inc_w} is vital in various inductive arguments given below. Note that,
moreover, $\lambda \in \Int W_e(T)$ implies that there exists an infinite rank (orthogonal) projection $P$ such that
$PTP=\lambda P.$ Hence $\lambda \in \Int W_e(T)$ yield infinite-dimensional $\lambda$-compressions of $T.$
The basic properties of $W_e(T)$ can be found e.g. in \cite[Chapter 7.34]{Bonsall}, see also \cite{Anderson}, \cite{Fillmore} and \cite{Stampfli}. Some of their analogons for tuples of bounded operators are proved in \cite{Li-Poon}. A 
 unified approach to essential numerical range for operator tuples and other related numerical ranges
has been provided in  \cite[Section 5]{MT_LMS}. 
 
Finally, for auxiliary estimates, we will need the next  plank type result from \cite[Theorem 2]{Ball}.
\begin{theorem}\label{plank} Let $(u_j)_{j=1}^{n} \subset H$ be a tuple of unit vectors, and let 
 $(a_j)_{j=1}^{n} \subset [0,\infty)$ be such that $\sum_{j=1}^n a_j^2=1.$
Then there exists a unit vector $v \in H$ such that 
$|\langle v, u_j \rangle| \ge a_j$
for all $j.$
\end{theorem}
Note that, in fact, plank type results similar to the theorem above motivated their non-commutative generalisations in \cite{Lust},
e.g. Theorem \ref{lust} mentioned above and its more general versions.

\subsection{Results}

Extending and complementing various results on operator diagonals from the literature, 
it was proved in \cite[Corollary 4.3]{MT} that for every $T\in B(H)$ and every $(\la)_{n=1}^\infty \subset\Int W_e(T)$ satisfying  
\begin{equation}\label{blaschke}
\sum_{n=1}^\infty\dist\{\la_n,\partial W_e(T)\}=\infty,
\end{equation}
 there exists 
 an orthonormal basis $(u_n)_{n=1}^\infty$ in $H$ such that $\langle T u_n, u_n \rangle=\la_n, n \in \mathbb N.$ 
In fact, most of results in \cite{MT} were obtained in a slightly more general setting of operator tuples $\cT \in B(H)^m, m \in \mathbb N.$ Moreover, the results were formulated in terms of the relative interior of $W_e(\cT)$ 
instead of the interior of $W_e(\cT)$ to include  degenerate situations when the relative interior of $W_e(\cT)$ is non-empty, see Section \ref{final} for more on that.
The assumption \eqref{blaschke}, introduced in \cite{MT} and called there \emph{non-Blaschke type}, is close to be optimal and allows one to construct diagonals for general $T \in B(H).$ Moreover, it has operator-valued counterparts leading to construction
of operator-value diagonals and generalisations of the main results from \cite{B}.

It is natural to ask whether \eqref{blaschke} has further implications and can be used to preassign a part of the matrix of $T$ larger than the main diagonal.
Besides producing the main diagonal $(\langle T u_n, u_n\rangle)_{n=1}^{\infty}$  
by \eqref{blaschke}, we prove that under \eqref{blaschke} the matrix of $T$ can be sparsified by arranging zero matrix elements in any band 
outside of $(\langle T u_n, u_n\rangle)_{n=1}^{\infty}$. 
The result is opposite in a sense to the series of results concerning matrices with banded structure discussed in the introduction.
The obtained sparsification is rather mild, and in this sense the result is certainly  weaker than the results on banded structure.
On the other hand, it concerns general $T \in B(H)$ rather
than very specific classes of $B(H)$ (normal, unitary, selfadjoint), and it complements the results 
on banded matrix representations. 

\begin{theorem}\label{band}
Let $T\in B(H)$, and let $(\la_n)_{n=1}^\infty \subset\Int W_e(T)$ be such that $\sum_{n=1}^\infty\dist\{\la_n,\partial W_e(T)\}=\infty$. 
Then for every  $K \in \mathbb N$ there exists an orthonormal basis $(u_n)_{n=1}^\infty$ in $H$ such that
\begin{equation}\label{diag}
\langle Tu_n,u_n\rangle=\la_n,  \qquad n\in\NN,
\end{equation}
and
\begin{equation}\label{nondiag}
\langle Tu_n,u_j\rangle=0, \qquad 1\le |n-j|\le K.
\end{equation}
\end{theorem}

Recall that if $T$ is the unilateral (forward) shift on $H^2(\mathbb D),$ then $\Int W_e(T)=\mathbb D.$ If $(\la)_{n=1}^\infty \subset\mathbb D$ satisfies \eqref{diag}, then $\sum_{n=1}^\infty (1-|\la_n|)=\infty,$ see \cite[p. 862-63]{H} and \cite[p. 3577-3578]{MT}.
Thus the non-Blaschke type assumption \eqref{diag} is optimal.

Th next corollary of Theorem \ref{band} is straightforward. 

\begin{corollary}\label{window}
Let $T\in B(H)$ be such that $0\in\Int W_e(T).$ For every  $K\in \mathbb N$ there exists an orthonormal basis $(u_n)_{n=1}^{\infty}$ in $H$ such that
\begin{equation}\label{wind}
\langle Tu_n,u_j\rangle=0
\end{equation}
for all $n,j\in\NN$ with $|n-j|\le K$.
\end{corollary}

To distinguish the representations satisfying \eqref{wind}  from the representations with band structure, one may call the matrix of $T\in B(H)$ with respect to an orthonormal basis $(u_n)_{n=1}^{\infty}$ \emph{windowed} with window of width $K$ if $\langle Tu_n,u_j\rangle=0$ for all $n$ and $j$ such that $|n-j|\le K.$ In this terminology,  any $T \in B(H)$ with $0\in\Int W_e(T)$ allows a window matrix representation of any (finite) width.

If  $(\la_n)_{n=1}^\infty\subset\Int W_e(T)$ is well-separated from the boundary of $W_e(T),$ then 
by e.g. Theorem \ref{band} the sequence $(\la_n)_{n=1}^\infty$ is realizable as the main diagonal $(\langle Tu_n,u_{n}\rangle)_{n=1}^\infty.$ It appears that in this case  there are only size restrictions on the other two diagonals $(\langle Tu_n,u_{n+1}\rangle)_{n=1}^\infty$ and $(\langle Tu_{n+1},u_n\rangle)_{n=1}^{\infty}$ of $T$. The next statement supports this claim. For simplicity we formulate it for contractions (for non-contractions the result can be reformulated easily).
The result is a partial generalization of the main results in \cite{MT} and \cite{H} (as well as the results for specific classes of $T$
as in e. g. \cite{Arveson06}, \cite{Arveson_USA}, \cite{Bownik}, \cite{Bownik1}, \cite{Jasper}, and \cite{Kaftal}).
\begin{theorem}\label{three_diag}
Let $T\in B(H)$, $\|T\|\le 1,$ 
and let $\e>0$ be fixed. 
Let $(\la_n)_{n=1}^\infty\subset\Int W_e(T)$, and let $(\mu_n)_{n=1}^\infty, (\nu_n)_{n=1}^\infty\subset\CC$ satisfy 
$$\dist\{\la_n,\partial W_e(T)\}>2\e \qquad \text{and} \qquad \sup_{n \ge 1}(|\mu_n|, |\nu_n|)<\frac{\e\sqrt{\e}}{16}$$
 for all $n\in\NN$. 
Then there exists an orthonormal basis $(u_n)_{n=1}^\infty$ in $H$ such that for all $n \in \mathbb N:$
\begin{itemize}
\item [(i)]$\langle Tu_n,u_n\rangle=\la_n;$ \label{a}\\
\item [(ii)]$\langle Tu_n,u_{n+1}\rangle=\mu_n;$ \label{b}\\
\item [(iii)]$\langle Tu_{n+1},u_n\rangle=\nu_n.$ \label{c}
\end{itemize}
\end{theorem}

Next, instead of preassigning a part of $A_T$ we are interested in making the entries of $A_T$ vanishing fast enough 
at infinity. The method used in the proof of Theorems \ref{band} and \ref{three_diag} works here as well,
and a weaker assumption that $0\in W_e(T)$ will suffice for this purpose. 
In the following result one spreads out the estimate in \eqref{iii} 
over the whole of matrix $A_T$ of $T$ with respect to an appropriate orthonormal basis $(u_n)_{n=1}^\infty.$

\begin{theorem}\label{stoutt}
Let $T\in B(H).$ 
Then the following properties are equivalent.
\begin{itemize}
\item [(i)] $0\in W_e(T);$
\item [(ii)] 
 For any  $(a_n)_{n=1}^\infty\subset (0,\infty)$ satisfying $(a_n)_{n=1}^\infty\not \in \ell^1(\mathbb N)$ there exists an orthonormal basis $(u_n)_{n=1}^\infty$ in $H$ such that
\begin{equation}\label{bound}
|\langle Tu_n,u_j\rangle|\le \sqrt{a_n a_j}
\end{equation}
for all $n,j\in\NN$.
\end{itemize}
\end{theorem}
Note that for the diagonal elements $\langle Tu_n,u_n\rangle$ the estimate \eqref{bound} yields  
$|\langle Tu_n,u_n\rangle|\le a_n$ for all $n \in \mathbb N$, which is precisely Stout's condition \eqref{iii}.
Thus, \eqref{bound} can be considered as an extension of 
\eqref{iii}.
At the same time, the property (ii) in Theorem \ref{stoutt}  is clearly much stronger  than the ``small entry property'' in  Theorem \ref{stout}, (ii).
 Thus Theorem \ref{stoutt} is a vast generalization of Stout's Theorem \ref{stout}.
See also an open question on \cite[p. 45]{Stout} asking for a version of Theorem \ref{stoutt}.
Generalizations of \eqref{iii} in the framework of diagonals for operator tuples, and also to the context of operator-valued diagonals, can be found in \cite[Theorems 1.2 and 1.3]{MT}, see also \cite{MT_surv}.
%Note that \eqref{bound}  is essentially the same as in \eqref{iii} for the elements of $A_T$ near the main diagonal, 
%but weakens away from the main diagonal, e.g. one looses
%$\sqrt a_n$ in the first row or column.

By considering $T-\lambda$ for any $\lambda \in W_e(T)$ and applying Theorem \ref{stoutt} to $T-\lambda$
one gets the following corollary concerning arbitrary $T \in B(H).$
\begin{corollary}
Let $T\in B(H).$ Then for any $(a_n)_{n=1}^\infty\subset (0,\infty)$ satisfying $(a_n)_{n=1}^\infty\not \in \ell^1(\mathbb N)$ 
there exists an orthonormal basis $(u_n)_{n=1}^\infty$ in $H$ such that
$$
|\langle Tu_n,u_j\rangle|\le\sqrt{a_n a_j}
$$
for all $n,j\in\NN,$ $n\ne j$.
\end{corollary}

Choosing $(a_n)_{n=1}^{\infty} \in \ell^2(\mathbb N)\setminus \ell^1(\mathbb N)$ in Theorem \ref{stoutt}, we get the next statement 
complementing Corollary \ref{window} in the case when
e.g. $\Int W_e(T)=\emptyset.$ 
\begin{corollary}\label{window1}
Let $T\in B(H)$ and $\lambda  \in W_e(T).$   
Then for every  $K \in \mathbb N$ there exists an orthonormal basis $(u_n)_{n=1}^\infty$ in $H$  such that
\[
T=\lambda + W + S,
\]
where $W\in B(H)$ satisfies $\langle W u_n, u_j\rangle=0, |n-j| \le K,$  
and $S$ is a Hilbert-Schmidt operator on $H$.
\end{corollary}

Next we consider a problem opposite in a sense to the one addressed in Theorem \ref{stoutt}. For fixed operator $T \in B(H)$ we would like to find a basis $(u_n)_{n=1}^\infty$ giving rise to a matrix $A_T$ of $T$ with large entries $\langle Tu_n,u_j\rangle$. 
Since the task of  getting the lower bounds for $\langle Tu_n,u_j\rangle$ seems to be more demanding than the one concerning the 
upper bounds,
 we restrict ourselves to the polynomial scale of bounds. However, to illustrate the sharpness of our lower estimates
we provide the upper estimates as well. There is a certain gap between the two kinds of estimates, and we do not know whether
the gap can be removed, at least in a sense.
In terms of the size of constructed $A_T,$ the result given below is of course weaker than  Theorem \ref{lust}. However,
while in Theorem \ref{lust} one looks just for any operator satisfying \eqref{plank1}, our theorem produces an operator having large matrix entries and belonging to the unitary orbit of a fixed $T\in B(H)$ if $W_e(T)$ is large enough.
\begin{theorem}\label{polynom}
Let $T\in B(H)$ be an operator which is not of the form $T=\la + K$ for some $\la\in\CC$ and a compact operator $K\in B(H)$. Then there exist an orthonormal basis $(u_n)_{n=1}^\infty\subset H$ and strictly positive constants $c_1,c_2,$ and $d$ depending only on the diameter of $W_e(T)$ such that
\begin{equation}\label{11}
|\langle Tu_n,u_n\rangle|\ge d, \qquad n\in\NN,
\end{equation}
and
\begin{equation}\label{12}
\frac{c_1\min\{n,j\}^{1/2}}{\max\{n,j\}^{3/2}}\le
|\langle Tu_n,u_j\rangle|\le \frac{c_2}{\max\{n,j\}^{1/2}}
\end{equation}
for all $n,j\in\NN, n\ne j$.
\end{theorem}
It is probable that Theorem \ref{polynom} will be of value for the study of Schur multipliers, similarly as
Theorem \ref{lust} was used in \cite{Davidson}. However, we have not explored this direction.
Accidentally, the assumption of Theorem \ref{polynom} is equivalent to the fact that $T$ is a commutator (or that
$T$ is similar to an operator having an infinite-dimensional zero compression), see e.g. \cite[p. 440]{Anderson} and also \cite{Brown}.
Some methods and observations from the early days of commutator theory are similar in spirit to our proof of Theorem \ref{polynom}. 
However,  we are not aware of any deeper relations between our results and the commutator properties. 
On the other hand, estimates of the matrix elements were, in particular, important in the study of commutators in 
e.g. \cite[Section 3]{Brown},
 \cite{Mendel} and \cite{Mendel1}.

All of the results in this section stated above are nontrivial even for $T$ being a unilateral (forward) shift on $H^2(\mathbb D),$
and we do not see any simpler way to obtain them even in this very particular situation.
\subsection{Strategy}

It would be instructive and helpful to underline a general idea behind our arguments,
since it may possibly be modified and used in similar contexts as well.
While fine details of the proofs of above theorems differ from each other and require 
the corresponding adjustments, the general approach can be roughly described as follows. 

Let $T \in B(H)$ and assume that we are looking for an orthonormal basis $(u_n)_{n=1}^\infty$ such that $A_T=(\langle Tu_n, u_j \rangle)_{j,n=1}^\infty$ has certain property $(P)$. If $0 \in W_e(T)$, then for any subspace $M\subset H$ of finite codimension there exists a unit vector $e \in M$ such that $|\langle Te,e\rangle|<\epsilon.$ If moreover $0 \in \Int W_e(T)$, then  $\langle Te_\lambda,e_\lambda\rangle=\lambda$ for some unit vector $e_\lambda \in M$ 
and all $\lambda$ from 
a sufficiently small neighbourhood of zero. Using this one may construct inductively
the orthonormal sequence $(u_n)_{n=1}^{\infty}$ such that the sub-matrix $(\langle Tu_n, u_j \rangle)_{j,n=1}^\infty$ of $T$ has $(P).$
The construction is comparatively direct. 
Let $n\ge 1$ and suppose that we have constructed the vectors $u_1,\dots, u_{n-1}$ such that $A_{T, n-1}=(\langle Tu_i, u_j \rangle)_{j,i=1}^{n-1}$ satisfies $(P)$ (or its    
``approximation'').
The subspaces $M_{n-1}=\bigvee_{j=1}^{n-1}u_j$,  $T(M_{n-1})$ and $T^*(M_{n-1})$ span a subspace $\mathcal M_{n-1}$
of finite dimension in $H$. Choosing a unit vector $v_n \in \mathcal M_{n-1}^\perp$ according to the assumptions on $W_e(T)$ as above,
and employing simple finite-dimensional arguments, we may construct the required element $u_n$ as an appropriate linear combination of $u_j, Tu_j, T^*u_j, j=1, \dots, n-1,$
and $v_n.$

However, the main problem is to ensure that this sequence $(u_n)_{n=1}^\infty$ is a basis. This is achieved in the following way. First fix a countable sequence $(y_m)_{m=1}^{\infty}$ of unit vectors which generates all of $H$. Write $\NN=\bigcup_{m=1}^\infty A_m$ as a disjoint union of properly chosen sets $A_m$,
and for $m \in \mathbb N$ denote $m(n)$ the unique integer such that $n \in A_m.$ 
Arguing inductively as above, let $n\ge 1$, suppose that  $u_1,\dots, u_{n-1}$ with required $A_{T, n-1}$ have already been constructed.  Choose a unit vector $v_n$ orthogonal to $\mathcal M_{n-1}$ and define $u_n$ as a small perturbation of $v_n:$
$$
u_n:=\sqrt{1-c_n} v_n +c_n (I-P_{n-1})y_{m(n)},
$$ 
where $P_{n-1}$ is the orthogonal projection onto the subspace $M_{n-1}$ and $c_n\in (0,1)$ is ``small enough''. 

If we do it properly, then $A_{T, n}$ would still satisfy $(P)$, $u_n\perp \mathcal M_{n-1}$
and, recalling that $\|(I-P_{k})y_{m(n)}\|=\dist\{y_{m(n)},M_{k}\}$ for all $k,$
$$
\dist^2\{y_{m(n)}, M_n\}\le (1-c_n^2)\dist^2\{y_{m(n)},M_{n-1}\}.
$$

Moreover, since $M_n \subset M_{n+1}, n \ge 1,$ we clearly have $\dist\{y_{m},M_{k}\}\le \dist\{y_{m},M_{n}\}$ for $k \ge n.$
So if the numbers $c_n$ are chosen in such a way that $\sum_{n \ge 1} c_n^2=\infty,$ then 
$$
\lim_{n\to\infty}\ln \dist\{y_{m}, M_n\}\le -\lim_{n\to\infty} \sum_{j\le n\atop m(j)=m} c_j^2=-\infty
$$ 
for all $m\in\NN$. Hence $y_m\in\bigvee_{n=1}^\infty u_n$ for all $m$, and thus  $(u_n)_{n=1}^\infty$ will be the required orthonormal basis such that $A_T$ has $(P)$.

To simplify the exposition of the inductive arguments given below, we assume that a sum consisting of no terms (the ``empty sum'') equals zero,  the linear span of the empty set is $\{0\},$ and adopt other usual conventions on empty sets.

\section{Matrices with several given diagonals: proofs}
 
We start with a proof of Theorem \ref{band} being a generalization of \cite[Corollary 4.2]{MT}
in case of a single operator. (Concerning the setting of tuples see Section \ref{final}.)
The argument employed there is a good illustration of the strategy described above,
and  its variants will be used several times in the sequel.

\bigskip

{\it Proof of Theorem \ref{band}}\,\,
Let $(y_m)_{m=1}^\infty$ be a sequence of unit vectors in $H$ such that $\bigvee_{m=1}^\infty y_m=H$.

For $r=0,1,\dots,K$ let $B_r=\{n\in\NN: n=r \mod\, (K+1)\},$ and note that there exists $r_0\in\{0,\dots,K\}$ such that $\sum_{n\in B_{r_0}}\dist\{\la_n,\partial W_e(T)\}=\infty$.

Represent $B_{r_0}$ as $B_{r_0}=\bigcup_{m=1}^\infty A_m,$ where $A_m \cap A_n =\emptyset, m \neq n,$ and 
$$
\sum_{n\in A_m}\dist\{\la_n,\partial W_e(T)\}=\infty
$$
for all $m\in\NN$.

For $n\in B_{r_0}$ let $m(n)$ be the unique integer satisfying $n\in A_{m(n)}$.

We construct the vectors $u_n, n \ge 1,$ inductively. Let $n\ge 1$ and suppose that the vectors $u_1,\dots,u_{n-1}$ have already been constructed.

Suppose first that $n\notin B_{r_0}$. Let $\hat n=\min\{n'\in B_{r_0}: n'>n\}$.
Find a unit vector 
$$
u_n\in\{u_j,Tu_j,T^*u_j\quad j=1,\dots,n-1, y_{m(\hat n)}, Ty_{m(\hat n)}, T^*y_{m(\hat n)}\}^\perp
$$
such that $\langle Tu_n,u_n\rangle=\la_n$.
Clearly $u_n\perp u_1,\dots,u_{n-1}$,
$$
\langle Tu_j,u_n\rangle=0, \qquad j=\max\{1,n-K\},\dots,n-1,
$$
and
$$
\langle Tu_n,u_j\rangle=
\langle u_n,T^*u_j\rangle=0, \qquad j=\max\{1,n-K\},\dots,n-1.
$$
So $u_n$ satisfies \eqref{diag} and \eqref{nondiag}.

Suppose now that $n\in B_{r_0}.$ Let $P_{n-1}$ be the orthogonal projection onto the subspace $M_{n-1}:=\bigvee_{j=1}^{n-1} u_j$.

If $y_{m(n)}\in M_{n-1}$ then find a unit vector
$$
u_n\in\{u_j,Tu_j,T^*u_j: j=1,\dots,n-1\}^\perp
$$
such that $\langle Tu_n,u_n\rangle=\la_n$. Then $u_n$ satisfies \eqref{diag} and \eqref{nondiag} again.

Suppose that $y_{m(n)}\notin M_{n-1}$.
Set
$$
b_n=\frac{(I-P_{n-1})y_{m(n)}}{\|(I-P_{n-1})y_{m(n)}\|}.
$$
Let
$$\rho_n:=|\langle Tb_n,b_n\rangle-\la_n| \qquad \text{and} \qquad  \de_n:=\frac{1}{2}\dist\{\la_n,\partial W_e(T)\}.$$
If $\langle Tb_n,b_n\rangle=\la_n,$ then set $\mu_n=\la_n$. If $\langle Tb_n,b_n\rangle\ne\la_n,$ then let
$\mu_n\in\CC$ be the unique number satisfying
$$
\frac{\langle Tb_n,b_n\rangle-\la_n}{\rho_n}=\frac{\la_n-\mu_n}{\de_n}.
$$

Clearly $\mu_n\in \Int W_e(T)$. Using \eqref{inc_w},
find a unit vector
$$
v_n\in\{u_j,Tu_j,T^*u_j, j=1,\dots,n-1; b_n,Tb_n,T^*b_n\}^\perp
$$
satisfying $\langle Tv_n,v_n\rangle=\mu_n$. Set
$$
u_n=\sqrt{\frac{\rho_n}{\rho_n+\de_n}}\,v_n+\sqrt{\frac{\de_n}{\rho_n+\de_n}}\,b_n.
$$
Since $v_n\perp b_n$, we have $\|u_n\|=1$.
We have $u_n\perp \bigvee\{u_1,\dots,u_{n-1}\}$ and
\begin{align*}
\langle Tu_n,u_n\rangle=&
\frac{\rho_n}{\rho_n+\de_n}\langle Tv_n,v_n\rangle+\frac{\de_n}{\rho_n+\de_n}\langle Tb_n,b_n\rangle\\
=&
\frac{\rho_n\mu_n+\de_n\langle Tb_n,b_n\rangle}{\rho_n+\de_n}\\
=&
\frac{\rho_n}{\rho_n+\de_n}(\mu_n-\la_n)+\frac{\de_n}{\rho_n+\de_n}(\langle Tb_n,b_n\rangle-\la_n) +\la_n\\
=&\la_n.
\end{align*}
Suppose that $j$ satisfies $\max\{1,n-K\}\le j\le n-1$. Then $j,j+1,\dots,n-1\notin B_{r_0}$. By construction,
$\bigvee \{u_j,\dots,u_{n-1}\} \perp y_{m(n)}$. So
$$
b_n=\frac{(I-P_{n-1})y_{m(n)}}{\|(I-P_{n-1})y_{m(n)}\|}=
\frac{(I-P_{j-1})y_{m(n)}}{\|(I-P_{j-1})y_{m(n)}\|},
$$
hence $b_n$ is a linear combination of $u_1,\dots,u_{j-1},y_{m(n)}$, and so $Tu_j\perp b_n$.

Thus we have
$$
\langle Tu_j, u_n\rangle=
\sqrt{\frac{\rho_n}{\rho_n+\de_n}}\langle Tu_j,v_n\rangle+
\sqrt{\frac{\de_n}{\rho_n+\de_n}}\langle Tu_j,b_n\rangle=0.
$$
Similarly one proves that
$$
\langle Tu_n,u_j\rangle=
\langle u_n,T^*u_j\rangle=
\sqrt{\frac{\rho_n}{\rho_n+\de_n}}\langle v_n,T^*u_j\rangle +
\sqrt{\frac{\de_n}{\rho_n+\de_n}}\langle b_n,T^*u_j\rangle=0.
$$

If we continue the construction inductively, we construct an orthonormal system $(u_n)_{n=1}^{\infty}$ satisfying 
\eqref{diag} and \eqref{nondiag}.

It remains to show that it is a basis.
Note that for $n\in B_{r_0}$ we have 
\begin{align*}
\|(I-P_n)y_{m(n)}\|^2=&
\|(I-P_{n-1})y_{m(n)}\|^2-|\langle (I-P_{n-1})y_{m(n)},u_n\rangle|^2\\
=&
\|(I-P_{n-1})y_{m(n)}\|^2-\|(I-P_{n-1})y_{m(n)}\|^2\cdot|\langle b_n,u_n\rangle|^2\\
=&
\|(I-P_{n-1})y_{m(n)}\|^2\Bigl(1-\frac{\de_n}{\rho_n+\de_n}\Bigr)
\end{align*}
and
\begin{align*}
\ln\|(I-P_{n})y_{m(n)}\|^2=&
\ln\|(I-P_{n-1})y_{m(n)}\|^2+\ln\Bigl(1-\frac{\de_n}{\rho_n+\de_n}\Bigr)\\
\le&\ln\|(I-P_{n-1})y_{m(n)}\|^2-
\frac{\de_n}{\rho_n+\de_n}\\
\le&
\ln\|(I-P_{n-1})y_{m(n)}\|^2-\frac{\dist\{\la_n,\partial W_e(T)\}}{6\|T\|}.
\end{align*}
Now for fixed $m\in\NN$, we have
\begin{align*}
\lim_{n\to\infty}\ln\|(I-P_n)y_m\|^2\le& \lim_{n\to\infty}\sum_{j\le n\atop m(j)=m}\frac{-\dist\{\la_j,\partial W_e(T)\}}{6\|T\|}\\
=-&\sum_{n\in A_m}\frac{\dist\{\la_n,\partial W_e(T)\}}{6\|T\|}=-\infty.
\end{align*}
So $y_m\in\bigvee_{n=1}^\infty u_n$ for all $m$. Hence $(u_n)_{n=1}^{\infty}$ is an orthonormal basis,
and the proof is finished.
$\hfill \Box$

\bigskip

Under assumptions somewhat stronger than those in Theorem \ref{band},
 our techniques allows one to construct three diagonals of $T$ with  upper and lower sub-diagonals 
depending only on the $\sup$-norm of its main diagonal. To this aim, we first prove the next auxiliary lemma. 

\begin{lemma}\label{2d}
Let $T\in B(H)$, and let $\lambda \in\Int W_e(T)$ be such  that $$\dist\{\lambda ,\partial W_e(T)\}>\e>0.$$ Let $M\subset H$ be a subspace of finite codimension. Then there exists a unit vector $u\in M$ satisfying the following conditions:
\begin{itemize}
\item[(a)] $\langle Tu,u\rangle=\lambda$;

\item[(b)] if $w=Tu-\langle Tu,u\rangle u$, $w'=T^*u-\langle T^*u,u\rangle u$ and $\al,\beta\in\CC,$ then there exists $z\in\bigvee\{w,w'\}$ such that 
$$
\langle w,z\rangle=\al, \qquad \langle w',z\rangle=\beta \quad \text{and}
\quad \|z\|\le\frac{2(|\al|+|\beta|)}{\e}.
$$
\end{itemize}
\end{lemma}

\begin{proof}
Since $\lambda \pm \e$ and $\lambda \pm i\e\in\Int W_e(T),$ the property \eqref{inc_w} implies that there exists a unit vector $x_1\in M$
such that $\langle Tx_1,x_1\rangle=\lambda +\e$.

Similarly, there exists a unit vector $x_2\in M\cap\{x_1,Tx_1,T^*x_1\}^\perp$ such that $\langle Tx_2,x_2\rangle=\lambda+i\e$, and unit vectors $x_3\in M\cap\{x_j,Tx_j,T^*x_j:j=1,2\}^\perp$ and $x_4\in M\cap\{x_j,Tx_j,T^*x_j:j=1,2,3\}^\perp$ with $\langle Tx_3,x_3\rangle=\lambda -\e$ and $\langle Tx_4,x_4\rangle=\lambda-i\e$.

Let $$u=\frac{1}{2}(x_1+x_2+x_3+x_4).$$ Then $u\in M$, $\|u\|=1$ and 
$$
\langle Tu,u\rangle=
\frac{1}{4}\Bigl(\langle Tx_1,x_1\rangle+\langle Tx_2,x_2\rangle+\langle Tx_3,x_3\rangle+\langle Tx_4,x_4\rangle\bigr)=\lambda.
$$

Let 
$$
w=Tu-\langle Tu,u\rangle u \qquad  \text{and}\qquad w'=T^*u-\langle T^*u,u\rangle u.
$$
 Let $\eta\in\CC$. Then
\begin{align*}
\langle w+\eta w',x_1\rangle=&
\langle Tu,x_1\rangle-\langle Tu,u\rangle\cdot\langle u,x_1\rangle
+\eta\langle T^*u,x_1\rangle-\eta\langle T^*u,u\rangle\cdot\langle u,x_1\rangle\\
=&
\frac{1}{2}\langle Tx_1,x_1\rangle -\frac{\lambda}{2}+\frac{\eta\langle T^*x_1,x_1\rangle}{2}-\frac{\eta \bar \lambda}{2}\\
=&
\frac{\lambda +\e}{2}-\frac{\lambda}{2}+\frac{\eta(\bar \lambda +\e)}{2}-\frac{\eta \bar \lambda}{2}\\
=&
\frac{\e(1+\eta)}{2}
\end{align*}
and similarly,
\begin{align*}
\langle w+\eta w',x_2\rangle=&
\langle Tu,x_2\rangle-\langle Tu,u\rangle\cdot\langle u,x_2\rangle
+\eta\langle T^*u,x_2\rangle-\eta\langle T^*u,u\rangle\cdot\langle u,x_2\rangle\\
=&
\frac{1}{2}\langle Tx_2,x_2\rangle -\frac{\lambda}{2}+\frac{\eta\langle T^*x_2,x_2\rangle}{2}-\frac{\eta \bar \lambda}{2}\\
=&
\frac{\lambda +i\e}{2}-\frac{\lambda}{2}+\frac{\eta(\bar \lambda -i\e)}{2}-\frac{\eta \bar \lambda}{2}\\
=&
\frac{i\e(1-\eta)}{2}.
\end{align*}
So
\begin{align*}
\|w+\eta w'\|\ge&
\max\bigl\{|\langle w+\eta w',x_1\rangle|, |\langle w+\eta w',x_2\rangle|\bigr\}\\
=&\frac{\e}{2}\max\{|1+\eta|,|1-\eta|\}
\ge \frac{\e}{2}.
\end{align*}
In the same way, one can show $$\|\eta w+w'\|\ge\frac{\e}{2}$$ for all $\eta\in\CC$.

Let $L$ be the two-dimensional subspace generated by $w$ and $w'$. Let $P$ and $P'$ be the orthogonal projections from $L$ onto the one-dimensional subspaces generated by $w$ and $w'$, respectively.
We then have
$$
\|(I-P')w\|\ge\frac{\e}{2} \qquad \text{and} \qquad  \|(I-P)w'\|\ge\frac{\e}{2}.
$$

Let $\al,\beta\in\CC$, and let
$$
z=\frac{\al (I-P')w}{\|(I-P')w\|^2}+\frac{\beta (I-P)w'}{\|(I-P)w\|^2}.
$$
Then
$$
\langle w,z\rangle=\frac{\al \langle w, (I-P')w\rangle}{\|(I-P')w\|^2}=\al
\qquad
\text{and} \qquad
\langle w',z\rangle=\beta.
$$
Finally, 
$$
\|z\|\le \frac{|\al|}{\|(I-P')w\|}+\frac{|\beta|}{\|(I-P)w'\|}\le
\frac{2(|\al|+|\beta|)}{\e}.
$$
\end{proof}

The proof of Theorem \ref{three_diag} is similar to the proof of Theorem \ref{band},
but it is technically more demanding.

\bigskip

{\it Proof of Theorem \ref{three_diag}}\,\,
Note that the assumption $\|T\|\le 1$ implies that $|\e|\le 1$. 

We fix an orthonormal basis $(y_m)_{m=1}^\infty$ in $H$,
and construct the basis $(u_n)_{n=1}^{\infty}$ inductively.

For $n=1$, find a unit vector $u_1\in H$ with $\langle Tu_1,u_1\rangle=\la_1$ such that the vectors 
$$
w_1:=Tu_1-\langle Tu_1,u_1\rangle u_1 \qquad \text{and}  \qquad w'_1:=T^*u_1-\langle T^*u_1,u_1\rangle u_1
$$
 satisfy condition (ii) of Lemma \ref{2d}, i.e., for all $\al,\beta\in\CC$ there exists $z\in\bigvee\{w_1,w_1'\}$ such that
$$ \langle w_1,z\rangle=\al,  \qquad \langle w'_1,z\rangle=\beta, \quad \text{and}\quad  \|z\|\le\frac{2(|\al|+|\beta|)}{\epsilon}.$$
Set formally $v_1=u_1$, $z_1=b_1=0$.

Let $n\ge 2$ and suppose that the orthonormal vectors $u_1,\dots,u_{n-1}$ satisfying (i), (ii), and (iii) of Theorem \ref{three_diag} have already been constructed. To run the induction, we also assume that $u_{n-1}$ satisfies 
$$
u_{n-1}=v_{n-1}+z_{n-1}+b_{n-1},
$$ 
where $\|z_{n-1}\|\le\frac{\sqrt{\e}}{2}$, $\|b_{n-1}\|\le\frac{\e\sqrt{\e}}{32}$, and
$$
v_{n-1}\perp\{u_1,\dots,u_{n-2},z_{n-1},b_{n-1},Tz_{n-1},Tb_{n-1},T^*z_{n-1},T^*b_{n-1}\}.
$$
Moreover, if
$$w_{n-1}:=Tv_{n-1}-\langle Tv_{n-1},v_{n-1}\rangle v_{n-1} \quad \text{and} \quad  w'_{n-1}:=T^*v_{n-1}-\langle T^*v_{n-1},v_{n-1}\rangle v_{n-1},$$ then we suppose also that $w_{n-1}$ and $w'_{n-1}$ satisfy condition (b) of Lemma \ref{2d}, i.e., for all $\al,\beta\in\CC$ there exists $z\in\bigvee\{w_{n-1},w_{n-1}'\}$ with
$$ \langle w_{n-1},z\rangle=\al, \qquad \langle w'_{n-1},z\rangle=\beta \quad \text{and}\quad  \|z\|\le
\frac{2(|\al|+|\beta|)}{\epsilon}.$$

Denote by $P_{n-1}$ the orthogonal projection onto the subspace $M_{n-1}:=\bigvee_{j=1}^{n-1} u_j.$

Denote by $A_n$ the set of all positive integers $m$ such that $y_m\notin M_{n-1}$ and
$$
\sup\bigl\{|\langle(I-P_{n-1})y_m,z\rangle|:z\in\bigvee\{u_{n-1},Tu_{n-1},T^*u_{n-1}\},\,\, \|z\|=1\bigr\}$$
does not exceed $$\|(I-P_{n-1})y_m\|\cdot\frac{\e}{32}.$$

Since $y_m\to 0, m \to \infty,$ weakly, the set $A_n$ contains all but a finite number of $m\in\NN$, so in particular $A_n\ne\emptyset.$  

Let $m(n)$ be any number $m\in A_n$ minimizing the quantity $m+\card\bigl\{k: 2\le k\le n-1, m(k)=m\bigr\}.$
(If there are more than one such numbers, then fix any of them).

Let
$$
b_n=\frac{(I-P_{n-1})y_{m(n)}}{\|(I-P_{n-1})y_{m(n)}\|}\cdot\frac{\e\sqrt{\e}}{32}.
$$
Observe that
\begin{align*}
\|z_{n-1}+b_{n-1}\|^2\le&
\Bigl(\frac{\sqrt{\e}}{2}+\frac{\e\sqrt{\e}}{32}\Bigr)^2
\le \e\cdot\Bigl(\frac{1}{2}+\frac{1}{32}\Bigr)^2\le
\frac{8\e}{25}\le\frac{8}{25},\\
1-\|z_{n-1}+b_{n-1}\|^2\ge& \frac{17}{25},
\end{align*}
and
$$
\sqrt{1-\|z_{n-1}+b_{n-1}\|^2}\ge \frac{4}{5}.
$$

By Lemma \ref{2d}, there exists $z_n\in\bigvee\{w_{n-1},w'_{n-1}\}$ such that
\begin{align*}
\langle w_{n-1},z_n\rangle=&\frac{\mu_{n-1}-\langle Tu_{n-1},b_n\rangle}{\sqrt{1-\|z_{n-1}+b_{n-1}\|^2}},\\
\langle w'_{n-1},z_n\rangle=&\frac{\bar\nu_{n-1}-\langle T^*u_{n-1},b_n\rangle}{\sqrt{1-\|z_{n-1}+b_{n-1}\|^2}},
\end{align*}
and
\begin{align*}
\|z_n\|\le&\frac{2}{\e}\Bigl(\Bigl|\frac{\mu_{n-1}-\langle Tu_{n-1},b_n\rangle}{\sqrt{1-\|z_{n-1}+b_{n-1}\|^2}}\Bigr|+
\Bigl|\frac{\bar\nu_{n-1}-\langle T^*u_{n-1},b_n\rangle}{\sqrt{1-\|z_{n-1}+b_{n-1}\|^2}}\Bigr|\Bigr)\\
\le&
\frac{2}{\e}\cdot\frac{5}{4}\Bigl(\frac{2\e\sqrt{\e}}{16}+2\|b_n\|\Bigr)\le 5\sqrt{\e}\Bigl(\frac{1}{16}+\frac{1}{32}\Bigr)\\
\le& \frac{\sqrt{\e}}{2}.
\end{align*}
Note that as above, 
$$\|z_n+b_n\|^2\le\frac{8\e}{25}\le\frac{8}{25}, \qquad 1-\|z_n+b_n\|^2\ge\frac{17}{25} \quad \text{and}
\quad \sqrt{1-\|z_n+b_n\|^2}\ge\frac{4}{5}.$$

Set
$$
\la_n'=\frac{\la_n-\bigl\langle T(z_n+b_n),z_n+b_n\bigr\rangle}{1-\|z_n+b_n\|^2}.
$$
We have
\begin{align*}
|\la'_n-\la_n|\le&
\Bigl|1-\frac{1}{1-\|z_n+b_n\|^2}\Bigr|+\frac{\|z_n+b_n\|^2}{1-\|z_n+b_n\|^2}\\
=&
\frac{2\|z_n+b_n\|^2}{1-\|z_n+b_n\|^2}\\
\le& \frac{ 2\e\frac{8}{25}}{\frac{17}{25}}\le \e.
\end{align*}
So $\dist\{\la'_n,\partial W_e(T)\}>\e$. By Lemma \ref{2d} again, there exists a unit vector $v_n$ such that
\begin{align*}
v_n \perp& \bigvee_{j=1}^{n-1 }\{ u_j,Tu_j,T^*u_j\}, \\
v_n\perp&\{y_{m(n)}, Ty_{m(n)}, T^*y_{m(n)}\},\\
\{v_n,Tv_n,T^*v_n\}\perp&\{z_n,Tz_n,T^*z_n,b_n, Tb_n, T^*b_n\},\\
\langle Tv_n,v_n\rangle=&\la'_n,
\end{align*}
and the vectors 
$$w_n=Tv_n-\langle Tv_n,v_n\rangle v_n \qquad \text{and} \qquad
w'_n=T^*v_n-\langle T^*v_n,v_n\rangle v_n$$
 satisfy condition (b) of Lemma \ref{2d}.

Note that $w_n,w'_n\in\bigvee\{v_n,Tv_n.T^*v_n\}$ and so $w_n$ and $w'_n$ are orthogonal to $\bigvee \{z_n,b_n\}$. Similarly, $Tw_n,$ $T^*w_n,$ $Tw'_n,$ and $T^*w'_n$ are orthogonal 
to $\bigvee \{z_n,b_n\}$.

Let
$$
u_n=\sqrt{1-\|z_n+b_n\|^2} v_n+z_n+b_n.
$$
Since $v_n\perp \bigvee\{z_n, b_n\}$, we have $\|u_n\|=1$.

By definition, $v_n\perp M_{n-1}$ and $b_n\perp M_{n-1}$. Moreover, 
$$z_n\in\bigvee\{w_{n-1},w'_{n-1}\}\subset\{u_1,\dots,u_{n-2}\}^\perp.$$

Furthermore, $w_{n-1}$ and $w'_{n-1}$ are orthogonal to $v_{n-1}$ and
$$
\{w_{n-1},w'_{n-1}\}\subset \bigvee\{Tv_{n-1},T^*v_{n-1},v_{n-1}\}\subset\{z_{n-1},b_{n-1}\}^\perp.
$$
So $z_n\perp u_{n-1}$  and $u_n\perp \bigvee \{u_1,\dots,u_{n-1}\}$.

Moreover, 
$$
Tz_n\in\bigvee\{Tw_{n-1},Tw'_{n-1}\}\subset\bigvee\{Tv_{n-1},T^2v_{n-1},TT^*v_{n-1}\}\subset\{z_{n-1},b_{n-1}\}^\perp,
$$
and similarly, 
$$T^*z_n\perp \bigvee \{z_{n-1}, b_{n-1}\}.$$

Now we are ready to show that $u_n$ verifies conditions (i), (ii) and (iii) of Theorem \ref{three_diag}. We have
\begin{align*}
\langle Tu_n,u_n\rangle=&
(1-\|z_n+b_n\|^2)\langle Tv_n,v_n\rangle +\bigl\langle T(z_n+b_n),z_n+b_n\bigr\rangle\\
 =&
(1-\|z_n+b_n\|^2)\la_n'+\bigl\langle T(z_n+b_n),z_n+b_n\bigr\rangle=
\la_n.
\end{align*}
Similarly,
\begin{align*}
\langle Tu_{n-1},u_n\rangle=&
\langle Tu_{n-1},z_n\rangle+\langle Tu_{n-1},b_n\rangle\\
=&
\sqrt{1-\|z_{n-1}+b_{n-1}\|^2}\langle Tv_{n-1},z_n\rangle+\langle Tu_{n-1},b_n\rangle\\
=&\sqrt{1-\|z_{n-1}+b_{n-1}\|^2}\langle w_{n-1},z_n\rangle+\langle Tu_{n-1},b_n\rangle\\
=&\mu_{n-1},
\end{align*}
and 
\begin{align*}
\langle Tu_n,u_{n-1}\rangle=&
\langle Tz_n,u_{n-1}\rangle +\langle Tb_n,u_{n-1}\rangle\\
=&
\sqrt{1-\|z_{n-1}+b_{n-1}\|^2}\langle Tz_n,v_{n-1}\rangle+\langle Tb_n,u_{n-1}\rangle\\
=&\sqrt{1-\|z_{n-1}+b_{n-1}\|^2}\langle z_n,w'_{n-1}\rangle+\langle Tb_n,u_{n-1}\rangle\\
=&\nu_{n-1}.
\end{align*}

Constructing the vectors $u_n$ inductively as above, we obtain an orthonormal system  $(u_n)_{n=1}^{\infty}$ satisfying all of the conditions 
of Theorem \ref{three_diag}. It remains to show that it is a basis, i.e., that $\bigvee_{n=1}^\infty u_n$ contains all vectors $y_m,  m\in\NN$.
This is clearly true if $y_m$ belongs to the space $M_n$ for some $n$. If $y_m\notin\bigcup_{n=1}^\infty M_n$, then $m(n)=m$ for infinitely many values of $n$. 
Note that for every $n \ge 2$ 
\begin{align*}
\bigl|\langle (I-P_{n-1})y_{m(n)},u_n\rangle\bigr|=&
\bigl|\langle (I-P_{n-1})y_{m(n)},z_n+b_n\rangle\bigr|\\
\ge&
\bigl|\langle (I-P_{n-1})y_{m(n)},b_n\rangle\bigr|-\bigl|\langle (I-P_{n-1})y_{m(n)},z_n\rangle\bigr|\\
\ge& \|(I-P_{n-1})y_{m(n)}\|\cdot\frac{\e\sqrt{\e}}{32}\\
-&
\|(I-P_{n-1})y_{m(n)}\|\cdot\|z_n\|\cdot\frac{\e}{32}\\
\ge&
\|(I-P_{n-1})y_{m(n)}\|\cdot\frac{\e\sqrt{\e}}{64}
\end{align*}
since $m(n)\in A_n$,
and thus
\begin{align*}
\|(I-P_n)y_{m(n)}\|^2=&\|(I-P_{n-1})y_{m(n)}\|^2-|\langle (I-P_{n-1})y_{m(n)},u_n\rangle|^2\\
\le& \|(I-P_{n-1})y_{m(n)}\|^2\Bigl(1-\frac{\e^3}{2^{12}}\Bigr).
\end{align*}
Hence, taking into account that
$$\|(I-P_n)y_m\|\le\|(I-P_{n-1})y_m\|$$ for all $m\in\NN$, $m\ne m(n),$
we infer that
$$
\lim_{n\to\infty}\|(I-P_n)y_m\|^2\le
\lim_{r\to\infty} \Bigl(1-\frac{\e^3}{2^{12}}\Bigr)^r=0,
$$
so that $y_m\in\bigvee_{n=1}^\infty u_n.$ 

Thus $(u_n)_{n=1}^\infty$ is an orthonormal basis, and the statement is proved.

\section{Matrices with small entries: proofs}

In this section, we extend Stout's bound from Theorem \ref{stout}, (iii) by providing
a similar bound for all matrix elements of $T \in B(H)$ rather than merely the main diagonal of $T$
for an appropriate basis in $H.$ Clearly, this will also significantly improve Theorem \ref{stout}, (ii).   
To this aim, we need a simple lemma.
\begin{lemma}\label{lone}
Let $(a_n)_{n=1}^{\infty}\subset (0,\infty)$ be such $(a_n)_{n=1}^{\infty} \not \in \ell^1(\mathbb N)$. Then there exists 
 $(a'_n)_{n=1}^{\infty}\subset (0,\infty)$ such that $(a'_n)_{n=1}^{\infty} \not \in \ell^1(\mathbb N),$  $0<a'_n\le\max\{1,a_n\}$ for all $n\in \mathbb N,$ and $\lim_{n\to \infty}\frac{a'_n}{a_n}=0$.
\end{lemma}

\begin{proof}
Set $n_0=0$. We construct numbers $n_k, k\in\NN$ inductively.

If $k\in\NN$ and the numbers $n_0,n_1,\dots,n_{k-1}$ have already been constructed, then find $n_k>n_{k-1}$ such that
$$
\sum_{n=n_{k-1}+1}^{n_k}a_n\ge k.
$$
For $n_{k-1}+1\le j\le n_k$ then set $a'_n=\min\{1,k^{-1}a_n\}$. So $\sum_{n=n_{k-1}+1}^{n_k}a'_n\ge 1$.

If the numbers $a'_n$ are constructed in this way, then clearly $\sum_{j=1}^\infty a'_n=\infty$ and $\lim_{j\to\infty}\frac{a'_n}{a_n}=0.$
\end{proof}

{\it Proof of Theorem \ref{stoutt}}\,\, 
The implication (ii)$\Rightarrow$ (i) is straightforward, so it remains to prove that (i)$\Rightarrow$(ii).
To this aim, without loss of generality, we may assume that $\|T\| \le 1$.

Using Lemma \ref{lone}, find a sequence $(a'_n)_{n=1}^\infty$ such that $0<a'_n\le\min\{1,a_n\}$ for all $n$, $(a'_n)_{n=1}^\infty \notin\ell^1$ and $\lim_{n\to\infty}\frac{a'_n}{a_n}=0.$

Let $(y_m)_{m=1}^\infty$ be a sequence of unit vectors in $H$ such that $\bigvee_{m\in\NN}y_m=H$.

Find mutually disjoint sets $A_m\subset\NN$ such that $\bigcup_{m\in\NN}A_m=\NN$ and
$$
\sum_{n\in A_m}a'_n=\infty
$$
for all $m\in\NN$. We may and will assume that $m\in\bigcup_{k=1}^m A_k$ for all $m$.

For $n\in\NN$ denote by $m(n)$ the uniquely determined integer satisfying
$n\in A_{m(n)}$. 

Define also 
$$
d(n):=\min \left \{r \in \mathbb N, r\ge n: \frac{a'_k}{a_k}<a_n\hbox{ for all } k\ge r \right \},
$$
and note that
 $$m(n)\le n\le d(n)$$ 
for all $n\in\NN$.

We construct the orthonormal basis $(u_n)_{n=1}^{\infty}$ inductively. Let $n\ge 1$ and suppose that the vectors $u_1,u_2,\dots,u_{n-1}\in H$ have already been constructed.

Since $0 \in W_e(T),$ there exists a unit vector $v_n\in H$ orthogonal to the union of the sets 
$$
\{u_j,Tu_j,T^*u_j: j=1,\dots,n-1\}
$$
and
$$
\{y_j,Ty_j,T^*y_j: j=1,2,\dots, d(n)\},
$$
such that
$$
|\langle Tv_n,v_n\rangle|<\frac{a'_n}{2}.
$$
Let $P_{n-1}$ be the orthogonal projection onto the subspace $M_{n-1}:=\bigvee_{j=1}^{n-1}u_j$.
If $y_{m(n)}\in M_{n-1}$ then set $u_n=v_n$. Suppose that $y_{m(n)}\notin M_{n-1}$. 
Then let 
$$
w_n=\frac{(I-P_{n-1})y_{m(n)}}{\|(I-P_{n-1})y_{m(n)}\|} \quad \text{and} \quad c_n=\frac{\sqrt{a'_n}}{2},
$$
and  set
$$
u_n=\sqrt{1-c_n^2}v_n+c_n w_n.
$$
Note that $w_n\in\bigvee \{u_j:j\le n-1\}\vee \{y_{m(n)}\}$. So, by the choice of $v_n,$ we have $w_n\perp v_n$, and similarly, $Tv_n\perp w_n$ and $T^*v_n\perp w_n$. In particular, 
\[
\|u_n\|=1 \qquad \text{and} \qquad u_n\perp \bigvee \{u_1,\dots,u_{n-1}\}.
\] 

First we estimate the diagonal term $\langle Tu_n,u_n\rangle$:
\begin{align*}
|\langle Tu_n,u_n\rangle|=&
\bigl|(1-c_n^2)\langle Tv_n,v_n\rangle + c_n^2 \langle T w_n,w_n\rangle\bigr|\\
\le&
|\langle Tv_n,v_n\rangle|+c_n^2\\
\le&
\frac{a'_n}{2}+\frac{a'_n}{4}\\
<&a_n.
\end{align*}

To estimate the non-diagonal terms $\langle Tu_j,u_n \rangle$ and $\langle Tu_n,u_j\rangle$ for $1\le j\le n-1$ and 
$n\ge 2$,
we distinguish two cases:

1) If $n\ge d(j),$ then
$$
|\langle Tu_j,u_n\rangle=
|\langle Tu_j,c_n w_n\rangle|\le c_n=
\frac{\sqrt{a'_n}}{2}\le \frac{\sqrt{a_n a_j}}{2}<\sqrt{a_n a_j}.
$$
Similarly,
$$
|\langle Tu_n,u_j\rangle|=
|\langle T^*u_j,u_n\rangle|<\sqrt{a_n a_j},
$$
so that the bound \eqref{bound} holds in this case.

2) If the opposite case  $n<d(j)$ holds. Then
\begin{align*}
|\langle Tu_j,u_n\rangle|=&
c_n|\langle Tu_j, w_n\rangle|\\
\le&
c_n|\langle Tv_j, w_n\rangle|
+c_nc_j|\langle Tw_j, w_n\rangle|\\
\le&
c_n|\langle Tv_j, w_n\rangle|+c_nc_j.
\end{align*}
Let $L$ be the subspace generated by the vectors
$\{v_k:1\le k\le n-1\}$ and $y_{m(1)},\dots,y_{m(n)}$.
Since for every $1\le k \le n,$ we have (by induction)
$$
w_k\in\bigvee\{u_1,\dots,u_{k-1}, y_{m(k)}\}\subset \bigvee\{v_1,\dots,v_{k-1},w_1,\dots,w_{k-1},y_{m(k)}\}\subset L,
$$
it follows that $\{w_k: 1 \le k \le n\} \subset L$.

Note that the vectors $v_1,\dots,v_{n-1}$ are mutually orthogonal.
Indeed, 
if $k <  k' \le n-1$, then  
by the choice of $v_{k'}$, we have $v_{k'}\perp u_1,\dots,u_k$ and $v_{k'}\perp y_{m(k)}$, 
since $m(k)\le k<k'\le d(k')$. Moreover, $v_{k}$ is a linear combination of $u_1,\dots,u_{k}, y_{m(k)}$. 
So $v_k\perp v_{k'}$.
Thus the subspace $L\cap \{v_j\}^\perp$ contains the vectors $v_k, 1\le k\le n-1, k\ne j.$ By the choice of $v_j$, since $m(k)\le k\le n< d(j)$ for all $k\le n$, it follows that
$L\cap \{v_j\}^\perp$ contains also the vectors $y_{m(1)},\dots,y_{m(n)}.$ 

Arguing similarly, we infer that  $Tv_k\perp v_{k'}$ and $T^*v_k\perp v_{k'},$ for $k <k'\le n-1,$ and 
moreover, by the choice of $v_j$, we have $Tv_j\perp y_{m(k)}$ and $T^*v_j\perp y_{m(k)}$ for any $k$ such that $1 \le k \le n$.

Hence 
\[v_j\perp T(L\cap \{v_j\}^\perp) \qquad \text{and} \qquad  v_j\perp T^*(L\cap \{v_j\}^\perp).\]
Therefore,
$$
|\langle Tv_j,w_n\rangle|=
|\langle Tv_j,v_j\rangle|\cdot|\langle w_n,v_j\rangle|\le
|\langle Tv_j,v_j\rangle|\le\frac{a'_j}{2},
$$
and then
$$
|\langle Tu_j,u_n\rangle|\le
\frac{c_na'_j}{2}+c_nc_j=
\frac{\sqrt{a'_n} \cdot a'_j}{4}+\frac{\sqrt{a'_n a'_j}}{4}<
\sqrt{a'_na'_j}\le\sqrt{a_na_j}.
$$

Similarly,
$$
|\langle Tu_n,u_j\rangle|=
|\langle T^*u_j,u_n\rangle|\le\sqrt{a_na_j}.
$$

Thus, the two estimates above imply that \eqref{bound} holds in the second case too.

Constructing the vectors $u_n, n\in\NN,$ in this way, we get the orthonormal system $(u_n)_{n=1}^\infty$
such that 
$$
|\langle Tu_n,u_j\rangle|\le\sqrt{a_n a_j}
$$
for all $n, j\in\NN$. 

We claim that, moreover, $(u_n)_{n=1}^\infty$ is an (orthonormal) basis of $H$. Let $m\in\NN$ be fixed. 
Since
$$
\|(I-P_k)y_m\|\le \|(I-P_{k-1})y_m\|
$$
for all $k \in \mathbb N,$ and 
\begin{align*}
\|(I-P_n)y_{m(n)}\|^2=&
\|(I-P_{n-1})y_{m(n)}\|^2-|\langle (I-P_{n-1})y_{m(n)},u_n\rangle|^2\\
=&
\|(I-P_{n-1})y_{m(n)}\|^2-|\langle (I-P_{n-1})y_{m(n)},c_nw_n\rangle|^2\notag\\
=&
\|(I-P_{n-1})y_{m(n)}\|^2(1-c_n^2)\notag \\
=&
\|(I-P_{n-1})y_{m(n)}\|^2\left(1-\frac{a'_n}{4}\right),\notag
\end{align*}
we conclude that
$$
\|(I-P_k)y_m\|^2\le \|(I-P_{k-1})y_m\|^2\Bigl(1-\frac{a'_k}{4}\Bigr)
$$
for all $k\in A_m$.
So
$$
\|(I-P_k)y_m\|^2\le
\prod_{j \le k \atop j\in A_m}\Bigl(1-\frac{a'_j}{4}\Bigr)
$$
and
$$
\lim_{k\to\infty}\ln\|(I-P_k)y_m\|^2\le
\sum_{k\in A_m}\ln\Bigl(1-\frac{a'_k}{4}\Bigr)\le
-\sum_{k\in A_m}\frac{a'_k}{4}=-\infty.
$$
Hence $y_m\in\bigvee\{u_n: n\in
\NN\}$. Since $\bigvee_{m\in\NN}y_m=H$, the claim follows.
This finishes the proof. $\hfill \Box$

\section{Matrices with large entries: proofs}
In this section we proceed with the proof of Theorem \ref{polynom} producing a matrix with large entries for $T\in B(H)$
with $W_e(T)$ containing more than two points.

We will need the next lemma similar in spirit to considerations in \cite[Section 2]{Brown}, see also \cite{Salinas}.
Recall that $T\in B(H)$ is compact if and only if $W_e(T)=\{0\}$. So $T$ is of the form $T=\la I+K$ for some $\la\in\CC$ and a compact operator $K\in B(H)$ if and only if $W_e(T)$ is a singleton, i.e., the diameter 
$$
\diam (W_e(T))=\max\{|\la-\mu|: \la,\mu\in W_e(T)\}=0.
$$

\begin{lemma}\label{pearcy}
Let $T\in B(H)$ be an operator which is not of the form $T=\la I+K$ for some $\la\in\CC$ and a compact operator $K\in B(H)$.
Let 
\[ 0<C<\frac{\diam (W_e(T))}{4\sqrt{2}} \qquad \text{and}\qquad  0<D<\frac{\diam (W_e(T))}{4}.\]
 Then 
for any subspace $M\subset H$ of finite codimension there exists a unit vector $u\in M$ such that
\begin{align*}
|\langle Tu,u\rangle|\ge& D,\\
\|Tu-\langle Tu,u\rangle u\|\ge& C,
\end{align*}
and
$$
\|T^*u-\langle T^*u,u\rangle u\|\ge C.
$$
\end{lemma}

\begin{proof}
Since $T\ne \la I+K$, the set $W_e(T)$ contains at least two points. Let $\la,\nu\in W_e(T)$ satisfy $|\la-\nu|=\diam (W_e(T))$.
Without loss of generality, we may assume that $|\la|\ge|\nu|$. Set $\mu=\frac{\la+\nu}{2}$. Then $\mu\in W_e(T)$ since it is a convex set.

We have 
$$
|\la-\mu|=\frac{|\la-\nu|}{2}=\frac{\diam (W_e(T))}{2}
$$
and
$$
|\la+\mu|=\Bigl|\frac{3\la}{2}+\frac{\nu}{2}\Bigr|\ge
|\la|\ge\frac{\diam (W_e(T))}{2}.
$$

Let $C$ and $D$ satisfy 
\[
0< D< \frac{\diam (W_e(T))}{4} \qquad \text{and} \qquad 0<C<\frac{\diam (W_e(T))}{4\sqrt{2}}.
\]
Let $M\subset H$ be a subspace of a finite codimension.

Let $\e>0$ satisfy $\e<\frac{|\la+\mu|}{2}-D$ and $\e<\frac{ |\la-\mu|}{2}-C\sqrt{2}$.

Find a unit vector $x\in M$ such that $|\langle Tx,x\rangle-\la|<\e$.
Let $M'=M\cap\{x,Tx,T^*x\}^\perp$. Then $\codim M'<\infty$ and there exits a unit vector $y\in M'$ such that
$|\langle Ty,y\rangle-\mu|<\e$.

Set $$u=\frac{x+y}{\sqrt{2}}.$$ Clearly $u\in M$. Since $y\perp x$, we have $\|u\|=1$.

We have
$$
|\langle Tu,u\rangle|= \frac{|\langle Tx,x\rangle+\langle Ty,y\rangle|}{2}\ge
\frac{|\la+\mu|}{2}-\e>D.
$$
Furthermore,
\begin{align*}
\|Tu-\langle Tu,u\rangle u\|\ge& \bigl|\langle Tu,x\rangle-\langle Tu,u\rangle\cdot\langle u,x\rangle\bigr|
\ge
\Bigl|\frac{\langle Tx,x\rangle}{\sqrt{2}}-\frac{\langle Tu,u\rangle}{\sqrt{2}}\Bigr|\\
\ge& \frac{1}{\sqrt{2}}\frac{\bigl|\langle Tx,x\rangle-\langle Ty,y\rangle\bigr|}{2}\ge
\frac{|\la-\mu|}{2\sqrt{2}}-\frac{\e}{\sqrt{2}}>C,
\end{align*}
and similarly,
\begin{align*}
\|T^*u-\langle T^*u,u\rangle u\|
\ge
\frac{|\bar\la-\bar\mu|}{2\sqrt{2}}-\frac{\e}{\sqrt{2}}>C.
\end{align*}

\end{proof}

Now we are ready to prove Theorem \ref{polynom}. Apart from Lemma \ref{pearcy}, its proof will essentially rely 
on plank type Theorem \ref{plank}.

\bigskip
 
{\it Proof of Theorem \ref{polynom}}\,\, 
Without loss of generality we may assume that $\|T\|=1$.

Fix a sequence $(y_m)_{m=0}^\infty$ of unit vectors in $H$ such that $\bigvee_{m=0}^\infty y_m=H$.

For $n\in\NN$ denote by $m(n)$ the unique non-negative integer such that $n=2^{m(n)}(2k-1)$ for some $k\in\NN$.
(One may introduce the corresponding sets $A_m, m \ge 0,$ as in the proofs of our previous results, but 
in this case such a procedure seems to be not justified.)

Let $C$ and $D$ be the constants given by Lemma \ref{pearcy}, i.e., for every subspace $M\subset H$ of finite codimension there exists a unit vector $u\in M$ with 
$$|\langle Tu,u\rangle|\ge D,\quad 
\|Tu-\langle Tu,u\rangle u\|\ge C \quad \text{and}\quad  \|T^*u-\langle T^*u,u\rangle u\|\ge C, $$
where $C$ and $D$ are independent of $M.$

Let $d= 2^{-1}{D} $. Fix $a\ge 1 $ such that 
$$\frac{4}{a}\le D \qquad \text{and}\qquad \frac{54}{\sqrt{a}}\le C^2,$$
and set 
$$c_1=\frac{C}{3\sqrt{a}} \qquad \text{and}\qquad  c_2=\frac{1}{\sqrt{a}}.$$

We construct the vectors $u_n, n \ge 1,$ inductively. Let $n\ge 1$ and suppose we have constructed   
mutually orthogonal unit vectors $u_1,\dots,u_{n-1}$ satisfying 
\begin{equation}\label{(1)}
|\langle Tu_j,u_j\rangle|\ge d, \qquad 1\le j\le n-1,
\end{equation}
and
\begin{equation}\label{(2)}
\frac{c_1\min\{k,j\}^{1/2}}{\max\{k,j\}^{3/2}}\le
|\langle Tu_k,u_j\rangle|\le \frac{c_2}{\max\{k,j\}^{1/2}}
\end{equation}
for all $1\le k,j\le n-1$ such that $k\ne j,$ as well as,
\begin{align}
\|(I-P_k)Tu_j\|^2\ge\frac{C^2j}{2k},\qquad 1\le j\le k\le n-1 ,\label{(3)}\\
\|(I-P_k)T^*u_j\|^2\ge\frac{C^2j}{2k}, \qquad 1\le j\le k\le n-1,\label{(4)}
\end{align}
and
\begin{equation}
\|(I-P_{n-1})y_m\|^2\le \prod_{j\le n-1\atop m(j)=m} \Bigl(1-\frac{1}{2aj}\Bigr), \qquad m \ge 1,\label{(5)}
\end{equation}
where $P_k$ is the orthogonal projection onto the subspace $M_k:=\bigvee\{u_1,\dots,u_k\}$ for every  $1\le k \le n-1.$
(Note that if $2^m>n-1$ then the product in \eqref{(5)} is over the empty set and, by definition, it is equal to $1$. The inequality $\|(I-P_{n-1})y_m\|^2\le\|y_m\|^2=1$ is then satisfied automatically).

We construct the unit vector $u_n$ satisfying \eqref{(1)}--\eqref{(5)} in the following way.
Using Lemma \ref{pearcy}, find a unit vector $v_n$ such that
$$
v_n\perp\bigl\{u_j,Tu_j,T^*u_j, \,\, 1\le j\le n-1; y_{m(j)}, Ty_{m(j)}, T^*y_{m(j)}, \,\, 1\le j\le n\bigr\}
$$
and
\begin{align*}
|\langle Tv_n,v_n\rangle|\ge& D,\\
\|Tv_n-\langle Tv_n,v_n\rangle v_n\|\ge& C,\\
\|T^*v_n-\langle T^*v_n,v_n\rangle v_n\|\ge& C.
\end{align*}

Consider now the $(2n-1)$-tuple of vectors 
\[(I-P_{n-1})Tu_j, (I-P_{n-1})T^*u_j,\,  j=1,\dots,n-1, \quad \text{and} \quad 
(I-P_{n-1})y_{m(n)}.\]

By Theorem \ref{plank}, there exists a unit vector $z_n$ such that
$$
z_n\in\bigvee\left \{(I-P_{n-1})Tu_j, (I-P_{n-1})T^*u_j, \, 1 \le j \le n-1; (I-P_{n-1})y_{m(n)} \right \}
$$
and
\begin{align}
\bigl|\langle (I-P_{n-1})y_{m(n)},z_n\rangle\bigr|\ge& \frac{1}{\sqrt{2}}\cdot\|(I-P_{n-1})y_{m(n)}\|, \label{z1}\\
\bigl|\langle (I-P_{n-1})Tu_j,z_n\rangle\bigr|\ge& \frac{1}{2\sqrt{n}}\cdot\|(I-P_{n-1})Tu_j\|, \label{z2}
\\
\bigl|\langle (I-P_{n-1})T^*u_j,z_n\rangle\bigr|\ge& \frac{1}{2\sqrt{n}}\cdot\|(I-P_{n-1})T^*u_j\|.\label{z3}
\end{align}
for all  $1 \le j \le n-1$ (since $\Bigl(\frac{1}{\sqrt{2}}\Bigr)^2+ 2(n-1)\Bigl(\frac{1}{2\sqrt{n}}\Bigr)^2<1$).

Note that $z_n\in M_{n-1}^\perp$ and $v_n\perp z_n$. Moreover, $Tv_n\perp z_n$ and $T^*v_n\perp z_n$.

Set
$$
u_n=\frac{1}{\sqrt{an}} z_n+\sqrt{1-\frac{1}{an}} v_n.
$$
Clearly 
$$\|u_n\|=1 \qquad \text{and} \qquad u_n\perp \bigvee \{u_1,\dots,u_{n-1}\}.$$

Let us show that $u_n$ satisfies conditions \eqref{(1)}--\eqref{(5)}.

We have
\begin{align*}
|\langle Tu_n,u_n\rangle|=&
\Bigl| \Bigl(1-\frac{1}{an}\Bigr)\langle Tv_n,v_n\rangle +
\frac{1}{an}\langle Tz_n,z_n\rangle\Bigr|\\
\ge&
|\langle Tv_n,v_n\rangle|-\frac{1}{an}-
\frac{1}{an}
\\
\ge&
D-\frac{2}{a}\ge \frac{D}{2}=d.
\end{align*}
So $u_n$ satisfies \eqref{(1)}.

Estimating the non-diagonal term $\langle Tu_j,u_n\rangle$ for $j=1,\dots,n-1$ and $n\ge 2$, we have by \eqref{(3)} and \eqref{z2}:
\begin{align*}
|\langle Tu_j,u_n\rangle|
=&
\frac{1}{\sqrt{an}}|\langle Tu_j,z_n\rangle|=
\frac{1}{\sqrt{an}}\bigl|\langle (I-P_{n-1})Tu_j,z_n\rangle\bigr|\\
\ge&
\frac{1}{\sqrt{an}}\cdot\frac{1}{2\sqrt{n}}\cdot\|(I-P_{n-1})Tu_j\|\\
\ge&
\frac{1}{2n\sqrt{a}}\cdot\frac{Cj^{1/2}}{\sqrt{2(n-1)}}.
\end{align*}
So
$$
|\langle Tu_j,u_n\rangle|\ge\frac{Cj^{1/2}}{3\sqrt{a}n^{3/2}}=\frac{c_1j^{1/2}}{n^{3/2}}.
$$
Obviously
$$
|\langle Tu_j,u_n\rangle|=\frac{1}{\sqrt{an}}|\langle Tu_j,z_n\rangle|\le
\frac{1}{\sqrt{an}}=
\frac{c_2}{n^{1/2}}.
$$
The inequalities
$$
\frac{c_1j^{1/2}}{n^{3/2}}\le
|\langle T^*u_j,u_n\rangle|=
|\langle Tu_n,u_j\rangle|\le
\frac{c_2}{n^{1/2}}
$$
for $j=1,\dots,n-1$ can be proved analogously, using \eqref{(4)} and \eqref{z3}. So \eqref{(2)} holds for all $1 \le j,k \le n$
such that $j\neq k$.

To prove \eqref{(3)} and \eqref{(4)} for $1 \le j,k \le n,$ let first $1 \le j \le n-1$. In view of \eqref{(3)} and \eqref{z2}, we have
\begin{align*}
\|(I-P_n)Tu_j\|^2=&
\bigl\|(I-P_{n-1})Tu_j\bigr\|^2-\bigl|\langle (I-P_{n-1})Tu_j,u_n\rangle\bigr|^2\\
=&
\|(I-P_{n-1})Tu_j\|^2-\frac{1}{an}\bigl|\langle (I-P_{n-1})Tu_j,z_n\rangle\bigr|^2\\
\ge&
\|(I-P_{n-1})Tu_j\|^2\Bigl(1-\frac{1}{an}\Bigr)\ge
\frac{C^2j}{2(n-1)}\cdot\frac{an-1}{an}\\
\ge&
\frac{C^2j}{2(n-1)}\cdot\frac{n-1}{n}=
\frac{C^2j}{2n}.
\end{align*}
Similarly, by \eqref{(4)} and \eqref{z3}, 
$$
\|(I-P_n)T^*u_j\|^2=
\bigl\|(I-P_{n-1})T^*u_j\bigr\|^2-\bigl|\langle (I-P_{n-1})T^*u_j,u_n\rangle\bigr|^2\ge
\frac{C^2j}{2n}
$$
for all $j=1,\dots,n-1$.

 Let now $j=n$ and estimate
$$
\|(I-P_n)Tu_n\|^2=
\|(I-P_{n-1})Tu_n\|^2-|\langle (I-P_{n-1})Tu_n,u_n\rangle|^2.
$$
We have
\begin{align*}
&\|(I-P_{n-1})Tu_n\|^2\\
=&
\Bigl\|\sqrt{1-\frac{1}{an}}Tv_n+\frac{1}{\sqrt{an}}(I-P_{n-1})Tz_n\Bigr\|^2\\
\ge&
\Bigl(1-\frac{1}{an}\Bigr)\|Tv_n\|^2-
\frac{2\sqrt{1-\frac{1}{an}}}{\sqrt{an}}\|Tv_n\|\cdot\|Tz_n\|+\frac{1}{an}\|(I-P_{n-1})Tz_n\|^2\\
\ge&
\|Tv_n\|^2-\frac{1}{an}-\frac{2}{\sqrt{an}}
\\
\ge&\|Tv_n\|^2-\frac{3}{\sqrt{a}}
\end{align*}
and
\begin{align*}
&\bigl|\langle (I-P_{n-1})Tu_n,u_n\rangle\bigr|=
|\langle Tu_n,u_n\rangle|\\
=&
\Bigl| \Bigl(1-\frac{1}{an}\Bigr)\langle Tv_n,v_n\rangle+
\frac{\sqrt{1-\frac{1}{an}}}{\sqrt{an}}\Bigl(\langle Tv_n,z_n\rangle+\langle Tz_n,v_n\rangle\Bigr)+
\frac{1}{an}\langle Tz_n,z_n\rangle\Bigr|\\
\le&
|\langle Tv_n,v_n\rangle|+\frac{1}{an}+\frac{2}{\sqrt{an}}+\frac{1}{an}\\
\le&
|\langle Tv_n,v_n\rangle| +\frac{4}{\sqrt{a}}.
\end{align*}
So,
in view of $\|Tv_n -\langle Tv_n,v_n\rangle v_n\|^2 \ge C^2,$ 
\begin{align*}
\|(I-P_n)Tu_n\|^2=&
\|(I-P_{n-1})Tu_n\|^2-\bigl|\langle (I-P_{n-1})T u_n,u_n\rangle\bigr|^2\\
\ge&
\|Tv_n\|^2-\frac{3}{\sqrt{a}}-\Bigl(|\langle Tv_n,v_n\rangle|+\frac{4}{\sqrt{a}}\Bigr)^2\\
\ge&
\bigl(\|Tv_n\|^2-|\langle Tv_n,v_n\rangle|^2\bigr)-\frac{3}{\sqrt{a}}-\frac{8}{\sqrt{a}}-\frac{16}{a}\\
\ge&
\|Tv_n -\langle Tv_n,v_n\rangle v_n\|^2-\frac{27}{\sqrt{a}}\\
\ge& \frac{C^2}{2}.
\end{align*}

The inequality
$$
\|(I-P_n)T^*u_n\|^2\ge\frac{C^2}{2}
$$
can be proved similarly. So $u_n$ satisfies \eqref{(3)} and \eqref{(4)} for all $1 \le j, k \le n.$

Given \eqref{(5)}, it remains to prove that in this case
\begin{equation}\label{(6)}
\|(I-P_{n})y_m\|^2\le \prod_{j\le n\atop m(j)=m} \Bigl(1-\frac{1}{2aj}\Bigr), \qquad m \ge 1.
\end{equation}
Taking into account that 
$
\|(I-P_n)y_m\|\le \|(I-P_{n-1})y_m\|,
$
$n\in \mathbb N,$
 it suffices to assume that $m=m(n).$ Then by \eqref{z1}, we have
\begin{align*}
\|(I-P_n)y_{m(n)}\|^2=&
\|(I-P_{n-1})y_{m(n)}\|^2-\bigl|\langle (I-P_{n-1})y_{m(n)},u_n\rangle\bigr|^2\\
=&\|(I-P_{n-1})y_{m(n)}\|^2-\frac{1}{an}\bigl|\langle (I-P_{n-1})y_{m(n)},z_n\rangle\bigr|^2\\
\le&
\|(I-P_{n-1})y_{m(n)}\|^2\Bigl(1-\frac{1}{2an}\Bigr),
\end{align*}
which yields \eqref{(6)}.

So the vectors $u_1,\dots, u_n$ satisfy \eqref{(1)}--\eqref{(5)}.

If we continue the construction inductively, we obtain an orthonormal system $(u_n)_{n=1}^\infty$ satisfying \eqref{11} and \eqref{12}. It remains to show that $(u_n)_{n=1}^\infty$ is a basis.
Let $m\ge 0$ be fixed. We have
\begin{align*}
\lim_{n\to\infty}\ln\|(I-P_n)y_m\|^2\le&
\lim_{n\to\infty}\sum_{j\le n\atop m(j)=m} \ln\Bigl(1-\frac{1}{2aj}\Bigr)\\
=&\lim_{k\to\infty}\sum_{j=1}^k\ln\Bigl(1-\frac{1}{2a\cdot2^m(2j-1)}\Bigr)\\
\le&
-\sum_{j=1}^\infty\frac{1}{a\cdot2^{m+1}(2j-1)}=-\infty.
\end{align*}
Hence
$$
\lim_{n\to\infty}\|(I-P_n)y_m\|^2=0
$$
and $y_m\in\bigvee_{j=1}^\infty u_j$. Since $\bigvee_{m=0}^\infty y_m=H$, we conclude that
$(u_n)_{n=1}^\infty$ is an orthonormal basis.
This finishes the proof.
$\hfill \Box$

\section{Final remarks}\label{final}
Note that Theorems \ref{band}, \ref{three_diag} and \ref{stoutt} have their counterparts for tuples
of bounded linear operators $\cT=(T_1, \dots, T_m) \in B(H)^m,$ $m \in \mathbb N.$ Moreover, their versions for tuples of selfadjoint operators
can be formulated under more general assumptions by replacing   the interior of the essential numerical range  $W_e(\cT)$ with its appropriate relative interior.
We have decided to present their single operator versions to simplify the presentation and to illustrate
the method rather than its fine technicalities.

To give a flavor of results one can obtain on this way,
recall that for $\cT \in B(H)^m$ the essential numerical range $W_e(\cT)$ of $\cT$
is defined as the set of all $m$-tuples $\alpha=(\alpha_1,\dots,\alpha_m)\in\CC^m$ such that there exists an orthonormal sequence (or, equivalently, basis) 
$(u_n)_{n=1}^\infty$ in $H$ satisfying
$$
\lim_{n\to\infty}\langle T_k u_n,u_n\rangle=\alpha_k
$$
for all $k=1,\dots,m$. 
This definition is completely analogous to the one given in Section \ref{results} for $m=1.$
As for $T \in B(H)$,  for every  $\cT \in B(H)^m$ the essential numerical range set $W_e(\cT)$ is nonempty, compact and convex.
For the theory of essential numerical ranges of operator tuples, one may consult e.g. \cite{Li-Poon}, \cite{MT_LMS}, \cite{Muller_Studia}
and \cite{MT_surv}.

Arguing as in the case $m=1$, cf. \cite{MT} for a ``tuple argument'', one gets the following statement.
\begin{theorem}\label{tuples}
Let $\cT\in B(H)^m$,  let $(\la_n)_{n=1}^\infty \subset\Int W_e(\cT)$ be such that $\sum_{n=1}^\infty\dist\{\la_n,\partial W_e(\cT)\}=\infty$,
and let $K\in \mathbb N$ be fixed. 
Then  there exists an orthonormal basis $(u_n)_{n=1}^\infty \subset H$ such that
\begin{equation*}
\langle \cT u_n,u_n\rangle=\la_n,  \qquad n\in\NN,
\end{equation*}
and
\begin{equation*}
\langle \cT u_n, u_j\rangle=0, \qquad 1\le |n-j|\le K.
\end{equation*}
\end{theorem} 
The formulation of a tuples analogue of Theorem \ref{stoutt} and the proof of that analogue are also direct adaptations of their
single operator versions.
\begin{theorem}\label{stout_t}
Let $\cT=(T_1,\dots,T_m)\in B(H)^m.$ Then the following statements are equivalent.
\begin{itemize}
\item [(i)]  $0\in W_e(\cT);$
\item [(ii)] For any  $(a_n)_{n=1}^\infty \subset (0,\infty)$  satisfying  $(a_n)_{n=1}^\infty\not \in \ell^1(\mathbb N)$ there exists an orthonormal basis $(u_n)_{n=1}^\infty$ in $H$ such that
\begin{equation*}
|\langle T_k u_n,u_j\rangle|\le \sqrt{a_n a_j}
\end{equation*}
for all $n,j\in\NN$ and $k=1,\dots,m$.
\end{itemize}
\end{theorem}
Theorem \ref{stout_t} implies, in particular, the following curious corollary. If $T \in B(H)$ and $T$ is boundedly invertible, then usually
$T^{-1}$ is ``large'' (in a sense) if $T$ is ``small'', and vice versa. However, sometimes one has a good control on both $T$ and $T^{-1}$
as in the result given below. 
Recall that the joint essential spectrum
$\sigma_e(\cT)$ for a commuting $\cT=(T_1, \dots, T_m) \in B(H)^m$ can be defined as the (Harte) spectrum of the $m$-tuple 
$(T_1 + K(H), . . . , T_m + K(H))$ in the
Calkin algebra $B(H)/K(H),$ where $K(H)$ denotes the ideal of all compact operators
on $H.$ It is well known, see e.g. \cite[p.123]{Dash}, that $\alpha=(\alpha_1, \dots, \alpha_m) \in \mathbb C^m$ belongs to $\sigma_e(\cT)$
if and only if there exists an orthonormal sequence $(u_n)_{n=1}^\infty \in H$ such that
either $\|T_k u_n - \alpha_k u_n\|\to 0,$ $n \to \infty,$ for every $1\le k \le m,$  or $\|T^*_k u_n - \alpha_k u_n\|\to 0,$
$n \to \infty,$ $ 1 \le k \le m.$ 
\begin{theorem}\label{inverse}
Let $T \in B(H)$ be such that $0 \not \in \sigma (T),$ and 
there exists a nonzero $z \in \mathbb C$ such that $\{z,-z\} \subset \sigma_e(T).$
Then for any  $(a_n)_{n=1}^\infty \subset (0,\infty)$  satisfying  $(a_n)_{n=1}^\infty\not \in \ell^1(\mathbb N)$ there exists an orthonormal basis $(u_n)_{n=1}^\infty$ in $H$ such that
\begin{equation*}
|\langle T u_n,u_j\rangle|\le \sqrt{a_n a_j} \qquad \text{and} \qquad |\langle T^{-1} u_n,u_j\rangle|\le \sqrt{a_n a_j}
\end{equation*}
for all $n,j\in\NN$.
\end{theorem}
 \begin{proof}

If $z \in \sigma_e(T),$
then there exists an orthonormal sequence $(u_n)_{n=1}^\infty$
such that either $\|T u_n -z u_n\| \to 0$ or $\|T^* u_n - z u_n \| \to 0, n \to\infty.$
Hence, either   $\|T^{-1} u_n -z^{-1} u_n\|=\|(z T)^{-1}(Tu_n -z u_n) \| \to 0,$
or $\|(T^*)^{-1} u_n - z^{-1} u_n \| \to 0, n \to \infty,$ respectively. 
This implies that $$(z,z^{-1})\in\sigma_e(T,T^{-1})\subset
 W_e(T,T^{-1}).$$
 Similarly,
 $(-z,-z^{-1})\in W_e(T,T^{-1})$. So, by convexity of $W_e(T,T^{-1})$, we have
$(0,0)\in W_e(T,T^{-1})$.
Now the statement follows from Theorem \ref{stout_t}.
\end{proof}
\begin{remark}
Note that under the theorem's assumptions, \cite[Corollary 30.11]{Muller} directly implies that
$\sigma_e(\cT)=\{(z, z^{-1}): z \in \sigma_e(T)\},$ but we preferred to not use a heavy machinery from
\cite{Muller} here.
\end{remark}
If $T\in B(H)$ is selfadjoint, then $W_e(T) \subset \mathbb R$ and the assumptions of Theorem \ref{tuples} never hold for $T$.
However, by considering the interior $\Int W_{e, \mathbb R}(T)$  of $W_e(T)$ relative to $\mathbb R$ one can prove the next statement
similar to Theorem \ref{tuples} by essentially the same argument.
 \begin{theorem}\label{tuples1}
Let $T \in B(H)$ be selfadjoint,  let $(\la)_{n=1}^\infty \subset\Int W_{e, \mathbb R}(T)$ be such that $\sum_{n=1}^\infty\dist\{\la_n,\partial W_{e, \mathbb R}(\cT)\}=\infty$,
and let $K\in \mathbb N$ be fixed. 
Then  there exists an orthonormal basis $(u_n)_{n=1}^\infty \subset H$ such that
\begin{equation*}
\langle T u_n,u_n\rangle=\la_n,  \qquad n\in\NN,
\end{equation*}
and
\begin{equation*}
\langle T u_n, u_j\rangle=0, \qquad 1\le |n-j|\le K.
\end{equation*}
\end{theorem} 
Since  $W_{e, \mathbb R}(T)=\conv \sigma_e (T)$ if $T$ is selfadjoint,
one has in fact $W_{e, \mathbb R}(T)=[\inf \sigma_e(T), \sup \sigma_e(T)],$ and the statement above can be recasted in spectral terms. 
Observe that the statement is not trivial even if $T$ is an operator of multiplication by independent variable on $L^2([0,1]).$

To deal with degenerate situations as above, instead of an $m$-tuple $\cT=(T_1, \dots, T_m)\in B(H)^m,$ we may consider a $2m$-tuple 
\[\widetilde \cT=(\Re T_1, \Im T_1, \dots, \Re T_m, \Im T_m) \] of selfadjoint operators $\Re T_i$ and $\Im T_i, 1 \le i \le m,$ 
and identify  $W_e(\cT)\subset \mathbb C^m$ with $W_e(\widetilde \cT)\subset \mathbb R^{2m}.$
So 
to formulate a result for $\tilde \cT$ similar to Theorem \ref{tuples1} (or to Theorem \ref{stout_t}), it is necessary  to replace  $\Int W_{e, \mathbb R}(T)$ by the interior
of   $W_e(\widetilde \cT)$ with respect to the smallest affine subspace of $\mathbb R^{2m}$ containing  $W(\widetilde \cT).$
Since this does not require any new arguments apart from comparatively simple linear algebra, we omit a detailed discussion
of that more general setting,
and refer the interested reader to \cite{MT}.


\begin{thebibliography}{100}

\bibitem{Anderson}
J. H. Anderson and J. G.  Stampfli, \emph{Commutators and compressions,}
Israel J. Math. \textbf{10} (1971), 433--441.

\bibitem{Arveson06} W. Arveson and R. V. Kadison, \emph{Diagonals of self-adjoint operators,} Operator Theory, Operator Algebras, and Applications, Contemp. Math., 414, 2006, AMS, Providence, RI, 247--263.

 \bibitem{Arveson_USA} W. Arveson, \emph{Diagonals of normal operators with finite spectrum,} Proc. Natl. Acad. Sci. USA 
\textbf{104} (2007), 1152--1158.

\bibitem{Ball}  K. M. Ball, \emph{The complex plank problem,} Bull. London Math. Soc. 
\textbf{33} (2001), 433--442. 

\bibitem{Ball_H} K. Ball, \emph{Convex geometry and functional analysis,} Handbook of the geometry of Banach spaces, Vol. I, North-Holland, Amsterdam, 2001, 161--104.

\bibitem{Bonsall}  F. F. Bonsall and J. Duncan, \emph{Numerical ranges,} II, LMS Lecture Notes Series, \textbf{10}, Cambridge University Press, New York-London, 1973.

\bibitem{B} J.-C. Bourin, \emph{Compressions and pinchings,} J. Operator Theory \textbf{50} (2003), 211--220.

\bibitem{Bownik} M. Bownik and J. Jasper, \emph{Characterization of sequences of frame norms,} J. Reine Angew. Math. \textbf{654} 
(2011), 219--244.

\bibitem{Bownik1} M. Bownik and J. Jasper, \emph{The Schur-Horn theorem for operators with finite spec-
trum,} Trans. Am. Math. Soc. \textbf{367} (2015), 5099--5140.

\bibitem{Brown} A. Brown and C. Pearcy, \emph{Structure of commutators of operators,}
 Ann. of Math.  \textbf{82} (1965), 112--127.

\bibitem{BrownO} 
N. P. Brown and N. Ozawa, \emph{
$C^*$-algebras and finite-dimensional approximations,}
Graduate Studies in Math., \textbf{88},  AMS, Providence, RI, 2008.

\bibitem{CMV} M. J. Cantero, L. Moral, and L. Vel\'azquez, \emph{
Five-diagonal matrices and zeros of orthogonal polynomials on the unit circle,}
Linear Algebra Appl. \textbf{362} (2003), 29--56.

\bibitem{CMV1} M. J. Cantero, L. Moral, and L. Vel\'azquez, \emph{
Minimal representations of unitary operators and orthogonal polynomials on the unit circle,}
Linear Algebra Appl. \textbf{408} (2005), 40--65.

\bibitem{Dash} A. T. Dash, \emph{
Joint essential spectra,}
Pacific J. Math. \textbf{64} (1976), 119--128.

\bibitem{Davidson} 
K. R. Davidson and A. Donsig, \emph{Norms of Schur multipliers,} 
Illinois J. Math. \textbf{51} (2007),  743--766.

\bibitem{Mendel} M. David, \emph{On a certain type of commutator,} J. Math. Mech. \textbf{19} (1970), 665--680.

 \bibitem{Mendel1} M. David, \emph{On a certain type of commutators of operators,} Israel J. Math. \textbf{9} (1971), 34--42.

\bibitem{DoJo07} D. \v Z. Dokovi\'c and C. R. Johnson, 
\emph{Unitarily achievable zero patterns and traces of words in $A$ and $A^*$,} Linear Algebra Appl. \textbf{421} (2007), 63--68.

\bibitem{Douglas} R. G. Douglas and C. Pearcy, \emph{A note on quasitriangular operators,} Duke Math. J. \textbf{37} (1970), 177--188. 

\bibitem{Fan84} P. Fan, \emph{On the diagonal of an operator,} 
Trans. Amer. Math. Soc. \textbf{283} (1984), 239--251.

\bibitem{Fan87} P. Fan, C.-K. Fong and D. Herrero, \emph{
On zero-diagonal operators and traces,}
Proc. Amer. Math. Soc. \textbf{99} (1987),  445--451.

\bibitem{Fillmore} P. A. Fillmore, J. G. Stampfli, and J. P. Williams, \emph{On the essential numerical range, the essential spectrum, and a problem of Halmos,} Acta Sci. Math. (Szeged) \textbf{33} (1972), 179--192.

\bibitem{FoRaRo87} C. K. Fong, H. Radjavi, and P. Rosenthal, \emph{Norms for matrices and operators,}
 J. Operator Theory \textbf{18} (1987),  99--113.

 \bibitem{FoWu93} C. K. Fong and P. Y. Wu, \emph{Diagonal operators: dilation, sum and product,}
 Acta Sci. Math. (Szeged) \textbf{57} (1993), 125--138.

\bibitem{FoWu96} C. K. Fong and P. Y. Wu, \emph{Band-diagonal operators,} Linear Algebra Appl. \textbf{248} (1996), 185--204. 

\bibitem{Halmos} P. Halmos, \emph{Bounded integral operators,} Abstract 752-47-6, Notices AMS. \textbf{25}, A-
124 (1978).

\bibitem{Halmos-book} P. Halmos, \emph{A Hilbert space problem book,} Second ed., Graduate Texts in Math, \textbf{19}, Springer, New York-Berlin, 1982. 

\bibitem{Halmos_S} P. Halmos and V. S. Sunder, \emph{Bounded integral operators on $L^2$ spaces,} Ergebnisse der Math., \textbf{96}, Springer, Berlin-New York, 1978.

\bibitem{He-Wo} D. Herrero and W. Wogen, \emph{On the multiplicity of $T\oplus T \oplus \dots \oplus T,$}
Rocky Mountain J. Math. \textbf{20} (1990),  445--466.

\bibitem{H} D. Herrero, \emph{The diagonal entries of a Hilbert space operator,} Rocky Mountain J. Math. \textbf{21} (1991), 857--864.

\bibitem{HoSc07} J.-P. Holbrook and Schoch, \emph{Moving zeros among matrices,} with an appendix by T. Ko\v sir and B. Al. Sethuraman, 
Linear Algebra Appl. \textbf{424} (2007), 83--95.
 
\bibitem{Jasper} J. Jasper, J. Loreaux, and G. Weiss, \emph{Thompson's theorem for compact operators and diagonals of unitary operators,} Indiana Univ. Math. J. \textbf{67} (2018),  1--27.
 
\bibitem{Kadison02a} R. V. Kadison, \emph{The Pythagorean Theorem I: the finite case,} 
Proc. Natl. Acad. Sci. USA
\textbf{99} (2002), 4178--4184.

\bibitem{Kadison02b} R. V. Kadison, \emph{The Pythagorean Theorem II: the infinite discrete case,} 
Proc. Natl. Acad. Sci. USA
\textbf{99} (2002), 5217--5222.

\bibitem{Kaftal} V. Kaftal and G. Weiss, \emph{
An infinite dimensional Schur-Horn theorem and majorization theory,}
J. Funct. Anal. \textbf{259} (2010),  3115--3162.

\bibitem{Kennedy}  M. Kennedy and P. Skoufranis, \emph{
The Schur-Horn problem for normal operators,}
Proc. Lond. Math. Soc. \textbf{111} (2015), 354--380.

\bibitem{Li-Poon} C.-K. Li and Y.-T. Poon, 
\emph{The joint essential numerical range of operators: convexity and related results,}
Studia Math. \textbf{194} (2009),  91--104.

\bibitem{LW20} J. Loreaux and G. Weiss, \emph{On diagonals of operators:
selfadjoint, normal and other classes,} Operator theory: themes and variations, Conf. Proc., Timisoara, 2016,
Theta Foundation, 2020, 193--214,  arXiv 1905.09987.

\bibitem{Lust} F. Lust-Piquard, \emph{On the coefficient problem: a version of the Kahane-Katznelson-de Leeuw theorem 
for spaces of matrices,} J. Funct. Anal. \textbf{149} (1997), 352--376.

\bibitem{Massey} P. Massey and M. Ravichandran, \emph{
Multivariable Schur-Horn theorems,}
Proc. Lond. Math. Soc. \textbf{112} (2016), 206--234.

\bibitem{MPT} A. V. Megretskiĭ, V. V. Peller, and S. R. Treil, 
\emph{The inverse spectral problem for self-adjoint Hankel operators,}
Acta Math. \textbf{174} (1995),  241--309.

\bibitem{Muller} V. M\" uller, \emph{Spectral theory of linear operators and spectral systems in Banach algebras,} Second ed., Operator Theory: Advances and Applications, \textbf{139}, Birkh\" auser, Basel, 2007.


\bibitem{Muller_Studia} V. M\" uller, \emph{The joint essential numerical range, compact perturbations, 
and the Olsen problem,} Studia Math. \textbf{197} (2010),  275--290. 

\bibitem{MT_JFA} V. M\"uller and Y. Tomilov, \emph{Circles in the spectrum and the geometry of orbits: A numerical ranges approach,} 
J. Funct. Anal. \textbf{274}(2018), 433--460.

\bibitem{MT_LMS} V. M\"uller and Y. Tomilov, \emph{Joint numerical ranges and compressions of powers of operators,} 
J. London Math. Soc. \textbf{99} (2019), 127--152.

\bibitem{MT} V. M\"uller and Y. Tomilov, \emph{Diagonals of operators and Blaschke's enigma,}
 Trans. Amer. Math. Soc. \textbf{372} (2019), 3565--3595.

\bibitem{MT_surv} V. M\"uller and Y. Tomilov, \emph{Joint numerical ranges: recent advances and applications,}
 Concr. Oper. \textbf{7} (2020),  133--154.

\bibitem{Pat}  S. Patnaik, S. Petrovic, and G.Weiss, \emph{Universal block tridiagonalization in $B(H)$ and beyond,} 
The mathematical legacy of Victor Lomonosov, de Gruyter, 2020, to appear. arXiv: 1905.00823.

\bibitem{RaRo} H. Radjavi and P. Rosenthal,
\emph{Matrices for operators and generators of $B(H)$,}
J. London Math. Soc.  \textbf{2} (1970), 557--560.

\bibitem{Salinas}  N. Salinas, \emph{On the $\eta$ function of Brown and Pearcy and the numerical function of an operator,} 
Canadian J. Math. \textbf{23} (1971), 565--578.

 \bibitem{Shulman} V. S. Shul'man, \emph{Multiplication operators and traces of commutators,} Investigations on linear operators and the theory of functions, XIII, Zap. Nauchn. Sem. Leningrad. Otdel. Mat. Inst. Steklov. (LOMI) \textbf{135} (1984), 182--194 (in Russian).

\bibitem{Simon} B. Simon, \emph{
CMV matrices: five years after,}
J. Comput. Appl. Math. \textbf{208} (2007),  120--154.

\bibitem{So78} A. R. Sourour, \emph{Operators with absolutely bounded matrices,} 
Math. Z. \textbf{162} (1978), 183--187.

\bibitem{So79} A. R. Sourour, \emph{
A short proof of Radjavi's theorem on self-commutators,}
Math. Ann. \textbf{239} (1979),  137--139.

\bibitem{Stampfli}  J. G. Stampfli and J. P. Williams, \emph{
Growth conditions and the numerical range in a Banach algebra,}
Tohoku Math. J. \textbf{20} (1968), 417--424.

\bibitem{Stout} Q. Stout, \emph{Schur products of operators and the essential numerical range,}
 Trans. Amer. Math. Soc. \textbf{264} (1981), 39--47.

\bibitem{Stout1}  Q. Stout, \emph{Schur multiplication on $B(\ell_p,\ell_q),$} 
J. Operator Theory \textbf{5} (1981),  231--243. 

\bibitem{Su78} V. S. Sunder, \emph{Absolutely bounded matrices,} 
Indiana Univ. Math. J. \textbf{27} (1978), 919--927. 

\bibitem{Su81} V. S. Sunder, \emph{
Unitary equivalence to integral operators,}
Pacific J. Math. \textbf{92} (1981),  211--215.

\bibitem{Weiss} G. Weiss, \emph{Commutators of Hilbert-Schmidt operators, II,} Integral Equations Operator Theory \textbf{3} (1980),  574--600.
\end{thebibliography}
\end{document}